\documentclass[10pt]{amsart}

\title{Global $N_\infty$-operads}
\author{Miguel Barrero}
\address{IMAPP, Radboud University Nijmegen, The Netherlands}
\email{m.barrero@math.ru.nl}

\usepackage[english]{babel}
\usepackage{mathtools}
\usepackage[T1]{fontenc}
\usepackage{amsmath, amsthm, amsfonts, mathrsfs, amsfonts, amssymb}
\usepackage{tikz-cd}
\usepackage{quiver}

\usepackage[shortlabels]{enumitem}

\usepackage[a4paper, textheight = 25cm, textwidth = 16cm, top = 2.5cm, centering]{geometry}

\usepackage[backend=biber, style=alphabetic, citestyle=alphabetic, maxnames=5,minnames=3, maxalphanames=5, doi=false, isbn=false, url=false, eprint=false]{biblatex}
\usepackage{chngcntr}
\counterwithout{equation}{section}
\usepackage{hyperref}

\usepackage{pict2e}
\makeatletter
\newcommand{\adjunction}[4]{%
  #1\colon #2%
  \mathrel{\vcenter{%
    \offinterlineskip\m@th
    \ialign{%
      \hfil$##$\hfil\cr
      \longrightharpoonup\cr
      \noalign{\kern-.3ex}
      \smallbot\cr
      \longleftharpoondown\cr
    }%
  }}%
  #3 \noloc #4%
}
\newcommand{\longrightharpoonup}{\relbar\joinrel\rightharpoonup}
\newcommand{\longleftharpoondown}{\leftharpoondown\joinrel\relbar}
\newcommand\noloc{%
  \nobreak
  \mspace{6mu plus 1mu}
  {:}
  \nonscript\mkern-\thinmuskip
  \mathpunct{}
  \mspace{2mu}
}
\newcommand{\smallbot}{%
  \begingroup\setlength\unitlength{.15em}%
  \begin{picture}(1,1)
  \roundcap
  \polyline(0,0)(1,0)
  \polyline(0.5,0)(0.5,1)
  \end{picture}%
  \endgroup
}
\makeatother

\usetikzlibrary{decorations.markings}
\tikzset{double line with arrow/.style args={#1,#2}{decorate,decoration={markings,%
mark=at position 0 with {\coordinate (ta-base-1) at (0,1pt);
\coordinate (ta-base-2) at (0,-1pt);},
mark=at position 1 with {\draw[#1] (ta-base-1) -- (0,1pt);
\draw[#2] (ta-base-2) -- (0,-1pt);
}}}}
\tikzset{Equal/.style={-,double line with arrow={-,-}}}

\newtheorem{thm}{Theorem}[section]
\newtheorem{coro}[thm]{Corollary}
\newtheorem{lemm}[thm]{Lemma}
\newtheorem{prop}[thm]{Proposition}
\newtheorem*{thmintro}{Theorem}
\theoremstyle{definition}
\newtheorem{defi}[thm]{Definition}

\theoremstyle{remark}
\newtheorem{constr}[thm]{Construction}
\theoremstyle{remark}

\theoremstyle{remark}
\newtheorem{rem}[thm]{Remark}
\theoremstyle{remark}
\newtheorem{ex}[thm]{Example}

\DeclareMathAlphabet{\mathpzc}{OT1}{pzc}{m}{it}

\newcommand{\Topcat}{\underline{\smash{\mathrm{Top}}}}

\newcommand{\id}{\mathrm{id}}
\newcommand{\RR}{\mathbb{R}}
\newcommand{\NN}{\mathbb{N}}
\DeclareMathOperator{\Ima}{Im}

\DeclareMathOperator{\Map}{Map}

\DeclareMathOperator{\Ho}{Ho}
\newcommand{\cat}{\mathscr{C}}

\newcommand{\Spc}{\underline{\smash{\mathpzc{Spc}}}}

\newcommand{\defeq}{\vcentcolon=}

\DeclareMathOperator*{\colim}{colim}

\newcommand{\opp}{\mathrm{op}}

\newcommand{\GTopcat}{\underline{\smash{\mathrm{GTop}}}}

\newcommand{\SnSpc}{\underline{\smash{\Sigma_n}}\underline{\smash{\mathpzc{Spc}}}}

\newcommand{\Lcat}{\underline{\mathrm{L}}}
\newcommand{\Opn}{\mathcal{O}_n}
\newcommand{\Ug}{\mathcal{U}_G}
\newcommand{\Uk}{\mathcal{U}_K}

\newcommand{\Op}{\mathcal{O}}

\newcommand{\Pop}{\mathcal{P}}

\newcommand{\FF}{\mathcal{F}}
\newcommand{\fat}{\mathscr{F}}
\newcommand{\sobspc}{\Sigma_\ast\textrm{-}\Spc}
\newcommand{\sob}{\Sigma_\ast\textrm{-}\cat}

\newcommand{\Nin}{N_\infty}

\newcommand{\Ningl}{\Nin^{\Lie}}
\newcommand{\Igl}{I_{\Lie}}
\newcommand{\IG}{I_G}

\newcommand{\IGc}{\overline{I}_G}

\newcommand{\Ie}{I_e}
\newcommand{\All}{\mathpzc{All}}
\newcommand{\None}{\emptyset}

\newcommand{\GHTopcat}{\underline{\smash{(\mathrm{G}\times\mathrm{H})}}    \underline{\smash{\mathrm{Top}}}}

\DeclareMathOperator{\res}{res}
\DeclareMathOperator{\tr}{tr}
\newcommand{\un}{\boldsymbol{1}}
\newcommand{\FMGn}{\FF^M_{G, n}}
\newcommand{\FNGn}{\FF^N_{G, n}}
\newcommand{\FMNGn}{\FF^{M \times N}_{G, n}}
\newcommand{\FMHn}{\FF^M_{H, n}}
\newcommand{\FMKn}{\FF^M_{K, n}}
\newcommand{\FOpGn}{\FF^\Op_{G, n}}
\newcommand{\FRGOpKn}{\FF^{R_G(\Op)}_{K, n}}
\newcommand{\FOpHn}{\FF^\Op_{H, n}}
\newcommand{\FOpHm}{\FF^\Op_{H, m}}
\newcommand{\FOpHnm}{\FF^\Op_{H, n + m}}
\newcommand{\FPopHn}{\FF^\Pop_{H, n}}
\newcommand{\OnTopcat}{\underline{\smash{\mathrm{O(n)} \mathrm{Top}}}}
\DeclareMathOperator{\catalan}{\mathbf{Cat}}
\newcommand{\Lie}{\mathpzc{Lie}}
\newcommand{\Ilie}{I_\Lie}

\DeclareMathOperator{\Sub}{Sub}

\DeclareMathOperator{\blk}{block}

\newcommand{\trSempty}{
	\begin{tikzpicture}[scale=0.3, every node/.style={draw,shape=circle,fill=black,inner sep=0pt,minimum size=1pt}]
		\node(S) at (0,0) {};
		\node(W) at (-1,1) {};
		\node(E) at (1,1) {};
		\node(N) at (0,2) {};
		\path[-]
		
		;
	\end{tikzpicture}
}
\newcommand{\trSsigma}{
	\begin{tikzpicture}[scale=0.3, every node/.style={draw,shape=circle,fill=black,inner sep=0pt,minimum size=1pt}]
		\node(S) at (0,0) {};
		\node(W) at (-1,1) {};
		\node(E) at (1,1) {};
		\node(N) at (0,2) {};
		\draw (0,0) -- (-1,1) ;
	\end{tikzpicture}
}
\newcommand{\trStau}{
	\begin{tikzpicture}[scale=0.3, every node/.style={draw,shape=circle,fill=black,inner sep=0pt,minimum size=1pt}]
		\node(S) at (0,0) {};
		\node(W) at (-1,1) {};
		\node(E) at (1,1) {};
		\node(N) at (0,2) {};
		\draw (0,0) -- (1,1) ;
	\end{tikzpicture}
}
\newcommand{\trStausigma}{
	\begin{tikzpicture}[scale=0.3, every node/.style={draw,shape=circle,fill=black,inner sep=0pt,minimum size=1pt}]
		\node(S) at (0,0) {};
		\node(W) at (-1,1) {};
		\node(E) at (1,1) {};
		\node(N) at (0,2) {};
		\draw (0,0) -- (-1,1) ;
		\draw (0,0) -- (1,1) ;
	\end{tikzpicture}
}
\newcommand{\trSsigmatoS}{
	\begin{tikzpicture}[scale=0.3, every node/.style={draw,shape=circle,fill=black,inner sep=0pt,minimum size=1pt}]
		\node(S) at (0,0) {};
		\node(W) at (-1,1) {};
		\node(E) at (1,1) {};
		\node(N) at (0,2) {};
		\draw (0,0) -- (1,1) ;
		\draw (-1,1) -- (0,2) ;
	\end{tikzpicture}
}
\newcommand{\trSall}{
	\begin{tikzpicture}[scale=0.3, every node/.style={draw,shape=circle,fill=black,inner sep=0pt,minimum size=1pt}]
		\node(S) at (0,0) {};
		\node(W) at (-1,1) {};
		\node(E) at (1,1) {};
		\node(N) at (0,2) {};
		\draw (0,0) -- (-1,1) ;
		\draw (0,0) -- (1,1) ;
		\draw (0,0) -- (0,2) ;
		\draw (-1,1) -- (0,2) ;
		\draw (1,1) -- (0,2) ;
	\end{tikzpicture}
}
\newcommand{\trStausigmaS}{
	\begin{tikzpicture}[scale=0.3, every node/.style={draw,shape=circle,fill=black,inner sep=0pt,minimum size=1pt}]
		\node(S) at (0,0) {};
		\node(W) at (-1,1) {};
		\node(E) at (1,1) {};
		\node(N) at (0,2) {};
		\draw (0,0) -- (-1,1) ;
		\draw (0,0) -- (1,1) ;
		\draw (0,0) -- (0,2) ;
	\end{tikzpicture}
}
\newcommand{\trSSsigmatoS}{
	\begin{tikzpicture}[scale=0.3, every node/.style={draw,shape=circle,fill=black,inner sep=0pt,minimum size=1pt}]
		\node(S) at (0,0) {};
		\node(W) at (-1,1) {};
		\node(E) at (1,1) {};
		\node(N) at (0,2) {};
		\draw (0,0) -- (-1,1) ;
		\draw (0,0) -- (1,1) ;
		\draw (0,0) -- (0,2) ;
		\draw (-1,1) -- (0,2) ;
	\end{tikzpicture}
}
\newcommand{\trSStautoS}{
	\begin{tikzpicture}[scale=0.3, every node/.style={draw,shape=circle,fill=black,inner sep=0pt,minimum size=1pt}]
		\node(S) at (0,0) {};
		\node(W) at (-1,1) {};
		\node(E) at (1,1) {};
		\node(N) at (0,2) {};
		\draw (0,0) -- (-1,1) ;
		\draw (0,0) -- (1,1) ;
		\draw (0,0) -- (0,2) ;
		\draw (1,1) -- (0,2) ;
	\end{tikzpicture}
}

\begin{document}
\begin{abstract}
    We define $\Nin$-operads in the globally equivariant setting and completely classify them. These global $\Nin$-operads model intermediate levels of equivariant commutativity in the global world, i. e. in the setting where objects have compatible actions by all compact Lie groups. We classify global $\Nin$-operads by giving an equivalence between the homotopy category of global $\Nin$-operads and the partially ordered set of global transfer systems, which are much simpler, algebraic objects. We also explore the relation between global $\Nin$-operads and $\Nin$-operads for a single group, recently introduced by Blumberg and Hill. One interesting consequence of our results is the fact that not all equivariant $\Nin$-operads can appear as restrictions of global $\Nin$-operads.
\end{abstract}
\maketitle
\tableofcontents

\section{Introduction}

In homotopy theory we often study objects with operations that are not strictly commutative, but instead \emph{commutative up to all higher homotopies}. Examples of these include $E_\infty$-ring spectra, $E_\infty$-spaces, or more general $E_\infty$-algebras. These $E_\infty$-structures can be codified by $E_\infty$-operads.

An operad in a symmetric monoidal category codifies a collection of abstract operations that compose in a prescribed way. An algebra over an operad is a representation of the abstract operations of the operad as concrete operations on some object. An operad in topological spaces is $E_\infty$ if its space of $n$-ary operations is weakly contractible. This implies that up to all higher homotopies an $E_\infty$-operad codifies a single $2$-ary operation that is commutative, associative and unital. All $E_\infty$-operads are equivalent, leading to a single notion of \emph{commutativity up to all higher homotopies}.

Equivariant homotopy theory has seen increased interest in recent years, in particular since the solution of the Kervaire invariant one problem by Hill, Hopkins and Ravenel in \cite{kervaire}. The $\Nin$-operads introduced by Blumberg and Hill in \cite{BLUMBERG2015658} are the equivariant analogs of $E_\infty$-operads, codyfying equivariant commutativity. There are non-equivalent $\Nin$-operads which codify different equivariant $E_\infty$-structures, that is, different levels of \emph{equivariant commutativity up to all higher homotopies}.

We illustrate this with an example. Let $C_2$ be the cyclic group with 2 elements. A $C_2$-equivariant commutative operation on a set $X$ with a $C_2$-action is a set function $m \colon X \times X \to X$ which is both $C_2$-equivariant and $\Sigma_2$-equivariant for the $\Sigma_2$-action that switches the two copies of $X$ in the source of $m$, making $m$ $(C_2 \times  \Sigma_2)$-equivariant. Let $H \leqslant C_2 \times \Sigma_2$ be the only subgroup of order $2$ not contained in either $C_2$ or $\Sigma_2$. By restriction $m$ is $H$-equivariant for the action on $X \times X$ that simultaneously acts on both copies of $X$ and flips them. Let $\Op$ be an operad in $C_2$-spaces which is non-equivariantly $E_\infty$. Thus it codifies a single operation which is commutative, associative and unital up to all higher homotopies. The operation codified by the operad $\Op$ is additionally $C_2$-equivariant if the $C_2$-fixed points of $\Op$ are weakly contractible. But this does not guarantee that the $H$-fixed points of $\Op$ are weakly contractible (i. e. that the operation codified by the operad is $H$-equivariant). In words, an operation can be separately \emph{equivariant} and \emph{non-equivariantly commutative up to all higher homotopies}. But being \emph{equivariantly fully commutative up to all higher homotopies} requires more. This observation suggests a variety of possible \emph{levels} of equivariant commutativity up to all higher homotopies.

The homotopy groups of algebras over equivariant $\Nin$-operads inherit additional structure. Let $G$ be a finite group, and let $X$ be a $G$-space. For each $K \leqslant H \leqslant G$ there is a restriction morphism $\res^H_K(X) \colon \pi^H_\ast(X) \to \pi^K_\ast(X)$ on equivariant homotopy groups. Each equivariant $\Nin$-operad encodes the existence of certain transfer morphisms $\tr^H_K(X) \colon \pi^K_\ast(X) \to \pi^H_\ast(X)$ on its algebras $X$, for some $K \leqslant H \leqslant G$. If this is the case we say that the abstract transfer from $K$ to $H$ is \emph{admissible} for the given $\Nin$-operad. These transfer and restriction morphisms assemble into an \emph{incomplete Mackey functor} \cite[Proposition~1.4]{mackey}. If $X$ is a $G$-spectrum and an algebra over an equivariant $\Nin$-operad we can instead construct norm maps that assemble into an \emph{incomplete Tambara functor} \cite[Theorem~4.14]{tambara}.

The collection of admissible abstract transfers for an equivariant $\Nin$-operad forms a relation on the set of subgroups of $G$. This relation has to be a partial order, and furthermore it has to be closed under conjugation and restriction to subgroups, due to the structure of the equivariant $\Nin$-operads. The relations on the subgroups of $G$ that satisfy these conditions are called \emph{$G$-transfer systems}. Blumberg and Hill conjectured that the associated $G$-transfer systems completely classify the equivalence types of equivariant $\Nin$-operads. The conjecture was proven separately by Gutiérrez and White in \cite{GutierrezWhite}, by Rubin in \cite{rubin1}, and by Bonventre and Pereira in \cite{BonventrePereira}.

In this paper we are interested in analogous operads encoding different levels of commutativity in the setting of \emph{global homotopy theory}. This is the homotopy theory of spaces which have compatible actions by all compact Lie groups. A globally equivariant space $X$ has an associated $G$-space $X_G$ for each compact Lie group $G$. These $G$-spaces are compatible for varying $G$ in the following strong sense. For any continuous homomorphism $\alpha \colon K \to H$ of compact Lie groups, there is a natural $K$-equivariant map $\alpha^\ast(X_H) \to X_K$, which is a $K$-homeomorphism if $\alpha$ is injective and closed. This in particular means that $X_K$ and $X_H$ are equivariantly homeomorphic if $K$ and $H$ are abstractly isomorphic. In contrast, if $X$ is a $G$-space, the restrictions of the $G$-action to subgroups of $G$ do not need to satisfy the strong compatibility condition just mentioned. The map $\alpha^\ast$ is only defined if $\alpha$ is given by restriction to a subgroup or conjugation by some element $g \in G$. We illustrate this difference with an example.

Consider the group $C_2 \times C_2$, and let $\iota_1 \colon C_2 \to C_2 \times C_2$ and $\iota_2 \colon C_2 \to C_2 \times C_2$ denote the embeddings into the first and second copy of $C_2$ respectively. If $X$ is a globally equivariant space, its underlying $(C_2 \times C_2)$-space $X_{C_2 \times C_2}$ cannot be equivalent to the discrete $(C_2 \times C_2)$-space $(C_2 \times C_2)/(C_2 \times \{e\})$. If this were the case, $X_{C_2}$ would need to be equivalent to both $\iota_1^\ast((C_2 \times C_2)/(C_2 \times \{e\}))$ the $C_2$-space which consists of two points with trivial $C_2$-action, and $\iota_2^\ast((C_2 \times C_2)/(C_2 \times \{e\}))$ which is a free $C_2$-orbit, and these are not equivalent.

This stronger compatibility condition means that global homotopy theory does not extend classical equivariant homotopy theory. But it also means that globally equivariant objects possess a richer structure than what can be interpreted from observing their $G$-actions separately. Many traditional $G$-equivariant constructions for varying $G$ assemble into a single globally equivariant construction, for example real and complex equivariant bordism and equivariant topological $K$-theory. Using this additional structure encoded by the global compatibility can lead to strong results in equivariant homotopy theory. Hausmann's proof \cite{hausmann2022global} of the equivariant analog of Quillen's theorem on formal group laws for abelian compact Lie groups is a recent example.

Consequently, algebras over globally equivariant operads are highly structured, and looking at this structure gives more information than what is available if one considers each group separately. Strictly commutative globally equivariant objects are the ultra-commutative ring spectra and ultra-commutative monoids of Schwede \cite{global}. But besides them, there are also globally equivariant objects with intermediate levels of commutativity. To model them we introduce in the present paper \emph{global $\Nin$-operads} (Definition~\ref{defiNinoperad}).

We use \emph{orthogonal spaces} and \emph{global equivalences} between them as a model for unstable global homotopy theory (see \cite[Chapter~1]{global}), and we define global $\Nin$-operads to be operads in orthogonal spaces that satisfy a certain condition similar to that of an equivariant $\Nin$-operad. In \cite{globaloperads} we constructed a model structure on the category of algebras over any operad in orthogonal spaces, we defined the equivalences between such operads, and we proved that these are precisely the morphisms of global operads that induce Quillen equivalences between the respective categories of algebras. The results of \cite{globaloperads} justify the definitions of global $\Nin$-operad and equivalence between them that we give in this paper.

The main result of this paper completely classifies global $\Nin$-operads with respect to the equivalences of global operads mentioned in the previous paragraph. We do this through \emph{global transfer systems}, a globally equivariant analog of the $G$-transfer systems that classify equivariant $\Nin$-operads. The precise statement is the following.

\begin{thmintro}[Theorem~\ref{thmmain}]
The functor $\mathpzc{T}$ that sends a global $\Nin$-operad to the global transfer system given by its admissible transfers is an equivalence between the homotopy category of global $\Nin$-operads $\Ho(\Ningl)$ and the partially ordered set $\Igl$ of global transfer systems with respect to inclusion.
\end{thmintro}

This classification theorem is analogous to the classification of equivariant $\Nin$-operads for a particular group $G$ via $G$-transfer systems mentioned earlier, but neither is a generalization of the other. As we explain later, in our proof of this theorem we use the $G$-equivariant result.

A global transfer system is a partial order on the class of compact Lie groups that refines the "finite-index subgroup" relation and is closed under pullback along continuous homomorphisms of compact Lie groups. This last requirement is directly related to the previously mentioned strong compatibility along arbitrary homomorphisms that globally equivariant objects possess. This requirement is also stronger than what is required of a $G$-transfer system for a single group $G$, namely closure under conjugation and restriction to subgroups. This means that not all $G$-transfer systems appear as restrictions of global transfer systems, or equivalently, that not all homotopy types of $G$-equivariant $\Nin$-operads appear as restrictions of global $\Nin$-operads.

There is the greatest global transfer system, and the associated global $\Nin$-operads are the global $E_\infty$-operads of \cite{globaloperads}. In the unstable setting algebras over these operads are equivalent to the previously mentioned ultra-commutative monoids. These are highly structured, and represent the highest level of globally equivariant commutativity. Algebras over intermediate global $\Nin$-operads represent intermediate levels of commutativity between full commutativity and the "naive" globally equivariant commutativity displayed by non-equivariant $E_\infty$-algebras given a trivial globally equivariant structure.

The proof of our main result is conceptually very simple, although for some of the steps we need to carefully consider the technical details. Some of the steps are similar to those of the $G$-equivariant case, but in others we need new arguments. We begin by checking that the admissible transfers of a global $\Nin$-operad assemble into a global transfer system. Then we prove that a morphism of global $\Nin$-operads is an equivalence if and only if the source and target have the same associated global transfer systems. Using this we obtain that the functor $\mathpzc{T} \colon \Ho(\Ningl) \to \Igl$ is fully faithful.

The biggest hurdle is constructing a global $\Nin$-operad representing a given global transfer system $T$. Restricting $T$ to the subgroups of a fixed group $G$ yields a $G$-transfer system $T_G$. By the equivalence between $\Nin$-operads and $G$-transfer systems in the equivariant case we can construct an equivariant $\Nin$-operad $\Op_G$ with associated $G$-transfer system $T_G$. Then we prove that "cofreely" inducing the $G$-action on $\Op_G$ to a globally equivariant structure yields a global $\Nin$-operad with global transfer system generated by $T_G$. We do this construction for $G= O(n)$ the orthogonal group and we take the product of the resulting operads for all $n$. Since any compact Lie group embeds into an orthogonal group, the collection of $T_{O(n)}$ for all $n$ contains all of the information of $T$. We use this to prove that this product of operads is a global $\Nin$-operad with global transfer system $T$, showing that the functor $\mathpzc{T} \colon \Ho(\Ningl) \to \Igl$ is an equivalence.

Our theorem reduces the problem of classifying global $\Nin$-operads to classifying global transfer systems, which amounts to a purely representation-theoretic question. Explicitly computing the partially ordered set of global transfer systems is a much easier task than computing the partially ordered set of $G$-transfer systems, which has been done for very few groups. In particular, the latter has been done for cyclic groups of order $p^n$ \cite{balchin2021ninftyoperads}, for cyclic groups of order $pqr$ for three different primes $p, q$ and $r$ \cite{balchinCpqr}, and for $\Sigma_3$  and the quaternion group \cite{Rubin2}. In contrast, in \cite{abeliantransfers} we computed the partially ordered set of global transfer systems for the family of all abelian compact Lie groups.

\subsection*{Acknowledgements}

The work presented in this paper is part of my PhD project. I would like to thank my PhD supervisor Magdalena Kędziorek for suggesting this topic and for her guidance and various enlightening conversations. I would also like to thank Niall Taggart and Eva Höning for their very helpful comments on an earlier draft of this paper.

\section{Background on global homotopy theory}

Global homotopy theory studies \emph{globally equivariant} objects, i. e. objects which have compatible actions by all compact Lie groups. This section serves as a recollection of basic concepts from global homotopy theory that will be used in this paper. For details we refer the reader to \cite[Chapter 1]{global}.

\begin{defi}
Let $\Topcat$ be the category of compactly generated weak Hausdorff spaces. Let $\Lcat$ be the $\Topcat$-enriched category with objects the finite-dimensional real inner product spaces, and morphisms the linear isometric embeddings between them. An \emph{orthogonal space} is a $\Topcat$-enriched functor $\Lcat \to \Topcat$. We use $\Spc$ to denote the $\Topcat$-enriched category of orthogonal spaces.
\end{defi}

An orthogonal space is a globally equivariant object in the following sense. For each compact Lie group $G$ we fix a complete orthogonal $G$-universe $\Ug$. Given an orthogonal space $X$ and a finite-dimensional subrepresentation $V \subset \Ug$, the space $X(V)$ inherits a $G$-action from the $G$-action on $V$ and the functoriality of $X$. The \emph{underlying $G$-space} $X(\Ug)$ is obtained by taking a colimit of these $G$-spaces $X(V)$ indexed on the finite-dimensional subrepresentations $V \subset \Ug$. A morphism of orthogonal spaces is a \emph{global equivalence} if it induces $G$-weak equivalences between the \emph{homotopy} colimits of the $G$-spaces $X(V)$ for every compact Lie group $G$. Orthogonal spaces and global equivalences are a model for unstable global homotopy theory. The technical details of these definitions can be found in~\cite[Definitions 1.1.1~and~1.1.2]{global}. Note that whenever we consider a representation in this paper, we mean an orthogonal representation. 

The underlying $G$-spaces of an orthogonal space $X$ are compatible in the sense that for each continuous homomorphism $\alpha\colon G \to K$ of compact Lie groups, there is a natural $G$-equivariant map from $\alpha^\ast(X(\Uk))$ to $X(\Ug)$, which is a natural $G$-equivariant homeomorphism if $\alpha$ is injective and closed. As mentioned in the introduction, this compatibility is stronger than the compatibility between the restrictions of a $G$-action on a space to the various subgroups of $G$.

There is a Day convolution product on the category of orthogonal spaces, called the \emph{box product} and denoted by $\boxtimes$, which is part of a closed symmetric monoidal structure (see \cite[Section 1.3]{global}). The box product $X \boxtimes Y$ receives the universal bimorphism from $(X, Y)$. The commutative monoids with respect to this box product are called \emph{ultra-commutative monoids}, and they possess a rich structure, which includes transfer morphisms for all finite index closed subgroups of compact Lie groups. Ultra-commutative monoids are equivalent to algebras over \emph{global $E_\infty$-operads}, which were introduced in \cite{globaloperads}. This equivalence is a consequence of \cite[Theorem II]{globaloperads}. Global $E_\infty$-operads represent the highest level of globally equivariant commutativity, and our goal in this paper is to study all possible levels of commutativity in this setting.

\section{\texorpdfstring{$\Nin$}{N-infinity}-symmetric sequences in \texorpdfstring{$\Spc$}{Spc}}

Operads in a general symmetric monoidal category $\cat$ can be defined as monoids in the category of \emph{symmetric sequences} in $\cat$ with respect to the \emph{composition} monoidal structure (see \cite[Sections 1 and 2]{fresse} for a detailed treatment). Specializing to orthogonal spaces, a symmetric sequence in $\Spc$ is a sequence $\{X_n\}_{n \in \NN}$ of orthogonal spaces with a $\Sigma_n$-action on $X_n$ for each $n$. Let $\sobspc$ denote the category of symmetric sequences in $\Spc$, which is complete and cocomplete because $\Spc$ is. In this section we study a certain class of symmetric sequences in $\Spc$. A global $\Nin$-operad (which we define in Section~\ref{sectionoperadtransfer}) will just be an operad in orthogonal spaces whose underlying symmetric sequence is in this class. We begin with several basic definitions.

\begin{defi}
\label{defigraphsub}
Let $G$ and $H$ be compact Lie groups. A closed subgroup $\Gamma \leqslant G \times K$ is a \emph{graph subgroup} if $\Gamma \cap (\{e_G\} \times K) = \{(e_G, e_K)\}$. We denote the set of all graph subgroups of $G \times K$ by $\fat(G, K)$. They are called graph subgroups because for any $\Gamma \in \fat(G, K)$ there exists a unique closed subgroup $H \leqslant G$ and a unique continuous homomorphism $\phi \colon H \to K$ such that $\Gamma$ is precisely the graph of $\phi$. When we need to be explicit we write $\Gamma_\phi$ for $\Gamma$.
\end{defi}

Throughout this paper we will only consider closed subgroups, and so we omit the word \emph{closed}.

We say that an orthogonal space $X \in \Spc$ is \emph{non-empty} if there exists a $V \in \Lcat$ such that $X(V) \neq \emptyset$. Equivalently, $X$ is non-empty if its underlying space $X(\mathcal{U}_e) = X(\RR^\infty)$ is non-empty.

\begin{defi}
\label{defiadmissiblegroups}
Let $M \in \sobspc$ be a symmetric sequence in $\Spc$. For each compact Lie group $G$ and each $n \geqslant 0$ let $\FMGn$ be the set of those graph subgroups $\Gamma$ of $G \times \Sigma_n$ such that there exists a $G$-representation $V$ with $M_n(V)^\Gamma \neq \emptyset$. We call these graph subgroups in $\FMGn$ the \emph{admissible} graph subgroups for $M$.

Equivalently $\Gamma \in \FMGn$ if and only if $(M_n(\Ug))^\Gamma \neq \emptyset$. The sets $\FMGn$ are the sets of graph subgroups $\Gamma$ such that $M$ has non-empty $\Gamma$-fixed points. Note that if $G$ is a connected compact Lie group, there are no non-trivial continuous homomorphisms from $G$ to $\Sigma_n$, but there are graph subgroups of $G \times \Sigma_n$ given by non-trivial homomorphisms from subgroups of $G$ to $\Sigma_n$.
\end{defi}

The fact that the subgroups in $\FMGn$ are those with non-empty fixed points imposes additional structure on the sets $\FMGn$.

\begin{lemm}[Structure of the sets $\FMGn$]
\label{lemmFMGn}
Let $M \in \sobspc$ be a symmetric sequence, with associated sets $\FMGn$ of admissible graph subgroups. For each compact Lie group $G$ and each $n \geqslant 0$ the following holds:
\begin{enumerate}[label = \roman*)]
    \item \label{lemmFMGn1} Let $H \leqslant G$ be a subgroup, so that $\fat(H, \Sigma_n) \subset \fat(G, \Sigma_n)$. Then a graph subgroup $\Gamma \in \fat(H, \Sigma_n)$ is in $\FMHn$ if and only if it is in $\FMGn$.
    \item \label{lemmFMGn2} If $f \colon K \to H$ is a continuous homomorphism and $\Gamma_\phi \in \FMHn$, then $\Gamma_{\phi \circ f} \in \FMKn$.
    \item If $\Gamma \in \FMGn$ and $\Gamma' \leqslant \Gamma$ is a subgroup, $\Gamma' \in \FMGn$
    \item \label{lemmFMGn4}For $(g, \sigma) \in G \times \Sigma_n$, if $\Gamma \in \FMGn$, then $(g, \sigma) \Gamma (g, \sigma)^{-1} \in \FMGn$.
    \item \label{lemmFMGn5} If $M_n$ is non-empty and $\Gamma_\phi = H \times \{e\} \in \fat(H, \Sigma_n)$, i. e. $\phi$ is the trivial homomorphism, then $\Gamma_\phi \in \FMHn$.
\end{enumerate}
\end{lemm}

\begin{proof}
\begin{enumerate}[label = \roman*)]
    \item Assume that $\Gamma \in \FMGn$. Then by definition there exists a $G$-representation $V$ such that $M_n(V)^\Gamma \neq \emptyset$. Consider the $H$-representation $\res^G_H V$, then $M_n(\res^G_H V)^\Gamma = M_n(V)^\Gamma \neq \emptyset$, so $\Gamma \in \FMHn$.
    
    Now assume instead that $\Gamma \in \FMHn$. Again by definition there is an $H$-representation $V$ such that $M_n(V)^\Gamma \neq \emptyset$. By \cite[III Theorem 4.5]{diecklie} there is a $G$-representation $W$ and an $H$-equivariant embedding $\psi \colon V \to \res^G_H W$. Then $M_n(\psi) \colon M_n(V) \to M_n(\res^G_H W)$ is $(H \times \Sigma_n)$-equivariant, and so $M_n(\res^G_H W)^\Gamma = M_n(W)^\Gamma \neq \emptyset$ and $\Gamma \in \FMGn$.
    \item If $\Gamma_\phi \in \FMHn$, then there is an $H$-representation $V$ such that $M_n(V)^{\Gamma_\phi} \neq \emptyset$. Consider the $K$-representation $f^\ast V$, then $M_n(f^\ast V)^{\Gamma_{\phi \circ f}} = M_n(V)^{\Gamma_{\phi|_{\Ima(f)}}} \supset M_n(V)^{\Gamma_\phi} \neq \emptyset$, so $\Gamma_{\phi \circ f} \in \FMKn$.
    \item If $\Gamma \in \FMGn$ then there is a $G$-representation $V$ such that $M_n(V)^\Gamma \neq \emptyset$, and $M_n(V)^\Gamma \subset M_n(V)^{\Gamma'}$.
    \item If $\Gamma \in \FMGn$, then there is an $H$-representation $V$ with some  $x \in M_n(V)^\Gamma$. Then $(g, \sigma) x \in M_n(V)^{(g, \sigma) \Gamma (g, \sigma)^{-1}}$ and so $(g, \sigma) \Gamma (g, \sigma)^{-1} \in \FMGn$.
    \item Since $M_n$ is non-empty, there is some $V \in \Lcat$ with $M_n(V) \neq \emptyset$. Consider the trivial $H$-action on $V$, then $M_n(V)^\Gamma \neq \emptyset$, and so $\Gamma \in \FMHn$. \qedhere
\end{enumerate}
\end{proof}

\begin{coro}
\label{corofamily}
For each symmetric sequence $M \in \sobspc$, each compact Lie group $G$, and each $n \geqslant 0$, if $M_n$ is non-empty $\FMGn$ is a family of subgroups of $G \times \Sigma_n$ contained in $\fat(G, \Sigma_n)$, closed under taking subgroups and conjugation, and which contains all subgroups of the form $H \times \{e\}$ for $H \leqslant G$.
\end{coro}

\begin{defi}
\label{defiequivnin}
An \emph{equivalence of symmetric sequences in $\Spc$} is a morphism $f \colon M \to N$ such that $f_n$ is a $\Sigma_n$-global equivalence for each $n$ (as defined in \cite[Definition 3.2]{globaloperads}). Explicitly, $f_n$ is a $\Sigma_n$-global equivalence if for each compact Lie group $G$, each graph subgroup $\Gamma \in \fat(G, \Sigma_n)$, each $G$-representation $V$ and $l \geqslant 0$, the following statement holds. For any pair of maps $(\alpha, \beta)$ as in

\[\begin{tikzcd}
\partial D^l \arrow[r, "\alpha"] \arrow[d, "i_l", hook] & M_n(V)^\Gamma \arrow[d, "f_n(V)^\Gamma"] \\
D^l \arrow[r, "\beta"]                                    & N_n(V)^\Gamma                       
\end{tikzcd}\]
there is a $G$-embedding $\psi\colon V \to W$ into a $G$-representation $W$ and a map $\lambda \colon D^l \to M_n(W)^\Gamma$ which satisfies that in the following diagram
\begin{equation}
\label{eqdefiequi}
\begin{tikzcd}
\partial D^l \arrow[r, "\alpha"] \arrow[d, "i_l", hook] & M_n(V)^\Gamma \arrow[r, "M_n(\psi)^\Gamma"] & M_n(W)^\Gamma \arrow[d, "f_n(W)^\Gamma"] \\
D^l \arrow[rru, "\lambda", dashed] \arrow[r, "\beta"']     & N_n(V)^\Gamma \arrow[r, "N_n(\psi)^\Gamma"] & N_n(W)^\Gamma 
\end{tikzcd}
\end{equation}
the upper left triangle commutes, and the lower right triangle commutes up to homotopy relative to $\partial D^l$.
\end{defi}

\begin{rem}
\label{remmodelss}
Let $\SnSpc$ be the category of $\Sigma_n$-objects in $\Spc$. In \cite[Appendix A]{globaloperads} we constructed a model structure on $\SnSpc$ (called the $\Sigma_n$-global model structure) for each $n \in \NN$. These model structures together give a model structure on the category $\sobspc$ of symmetric sequences in $\Spc$. The weak equivalences of this model structure are precisely the equivalences of symmetric sequences in $\Spc$ that we just defined. The reason why this definition of equivalence of symmetric sequences in $\Spc$ is the one we are interested in is because by \cite[Theorem II]{globaloperads} a morphism of operads in $(\Spc, \boxtimes)$ induces a Quillen equivalence between the respective categories of algebras if and only if it is an equivalence of symmetric sequences in $\Spc$ in the above sense. Consequently, a good model structure on operads in $\Spc$ will have these equivalences as weak equivalences (see also \cite[Remark 4.16]{globaloperads}). The existence of such a model structure is not needed anywhere in the present paper.
\end{rem}

\begin{rem}
\label{remmodelssfib}
The fibrations of this model structure on $\sobspc$ are the morphisms $f \colon M \to N$ such that each $f_n$ is a $\Sigma_n$-global fibration in the sense of \cite[Definition A.13]{globaloperads}. Therefore a symmetric sequence $M$ in $\Spc$ is fibrant if and only if for each $n$, each compact Lie group $K$, each embedding of faithful $K$-representations $\psi \colon V \to W$, and each graph subgroup $\Gamma \in \fat(K, \Sigma_n)$, the map $M_n(\psi)^\Gamma \colon M_n(V)^\Gamma \to M_n(W)^\Gamma$ is a weak homotopy equivalence.
\end{rem}

\begin{lemm}
\label{lemmcontain}
Let $f \colon M \to N$ be a morphism of symmetric sequences in $\Spc$. Then for each compact Lie group $G$ and each $n \geqslant 0$, $\FMGn \subset \FNGn$.
\end{lemm}
\begin{proof}
If $\Gamma \in \FMGn$, by definition there exists some $G$-representation $V$ with $M_n(V)^\Gamma \neq \emptyset$. Since $f_n(V)$ is $(G \times \Sigma_n)$-equivariant, $N_n(V)^\Gamma \neq \emptyset$ and so $\Gamma \in \FNGn$.
\end{proof}

The sets $\FMGn$ carry the information of which fixed points of a symmetric sequence in $\Spc$ are non-empty. We give now the definition of an $\Nin$-symmetric sequence in $\Spc$, which roughly requires the non-empty fixed points to be weakly contractible.

\begin{defi}[$\Nin$-symmetric sequences in $\Spc$]
\label{defiNin}
We say that $M \in \sobspc$ is an \emph{$\Nin$-symmetric sequence in $\Spc$} if for each $n \geqslant 0$ the orthogonal space $M_n$ is non-empty, and for each compact Lie group $G$ the set $\FMGn$ satisfies the following property. If $\Gamma \in \FMGn$, for any $G$-representation $V$, any $l \geqslant 0$, and any map $\alpha \colon \partial D^l \to M_n(V)^\Gamma$ there is a $G$-embedding $\psi \colon V \to W$ into a $G$-representation $W$ and a map $\lambda \colon D^l \to M_n(W)^\Gamma$ such that the following diagram commutes
\[\begin{tikzcd}
	{\partial D^l} & {M_n(V)^\Gamma} && {M_n(W)^\Gamma} \\
	{D^l.}
	\arrow["\alpha", from=1-1, to=1-2]
	\arrow["{i_l}"', from=1-1, to=2-1]
	\arrow["{M_n(\psi)^\Gamma}", from=1-2, to=1-4]
	\arrow["\lambda"', dashed, from=2-1, to=1-4]
\end{tikzcd}\]
\end{defi}

\begin{rem}
\label{remNinunderlying}
If each $M_n$ in $M \in \sobspc$ is $\Sigma_n$-free and a \emph{closed} orthogonal space (for each linear isometric embedding $\psi$ the map $M_n(\psi)$ is a closed embedding) then for each compact Lie group $G$ the underlying $G$-space $M_n(\Ug)$ is a universal space for the family $\FMGn$. Thus under these conditions the underlying $G$-symmetric sequence of an $\Nin$-symmetric sequence in $\Spc$ is $\Nin$ is the sense of Blumberg and Hill \cite[Definition 4.1]{BLUMBERG2015658} (for $G$ a compact Lie group see \cite{GutierrezWhite} instead).
\end{rem}

\begin{rem}
Note that in \cite{BLUMBERG2015658}, the conditions that are imposed on the families of subgroups associated to an $\Nin$-symmetric sequence for a fixed group $G$ only require it to be closed under precomposition by subgroup inclusions and isomorphisms given by conjugation by some $g \in G$. In the case of $\Nin$-symmetric sequences in $\Spc$, the families $\FMGn$ also have to be closed under precomposition by arbitrary homomorphisms, by Lemma~\ref{lemmFMGn}~(\ref{lemmFMGn2}).
\end{rem}

\begin{prop}
\label{propequininss}
Given two $\Nin$-symmetric sequences $M$ and $N$ in $\Spc$, a morphism between them $f \colon M \to N$ is an equivalence in $\sobspc$ if and only if $\FMGn= \FNGn$ for each compact Lie group $G$ and each $n \geqslant 0$.
\end{prop}
\begin{proof}
Fix some $n \geqslant 0$. Assume that $\FMGn= \FNGn$. If $\Gamma \in \fat(G, \Sigma_n)$ but $\Gamma \notin \FMGn= \FNGn$ then there exist no maps $\alpha$ and $\beta$ as in Definition~\ref{defiequivnin}, as their targets are empty, so there is nothing to check. If $\Gamma \in \FMGn= \FNGn$ then since $\Gamma \in \FMGn$ and $M$ is assumed to be $\Nin$, there exist $\psi \colon V \to W$ and $\lambda$ which make the upper left triangle of Diagram~\ref{eqdefiequi} of Definition~\ref{defiequivnin} commute. Define a map $F \colon D^l \times \{0\} \cup D^l \times \{1\} \cup \partial D^l \times [0, 1] \to N_n(W)^\Gamma$ using $f_n(W)^\Gamma \circ \lambda$, $N_n(\psi)^\Gamma \circ \beta$, and the fact that they coincide on $\partial D^l$. Then since $\Gamma \in \FNGn$ and $N$ is assumed to be $\Nin$, we can extend $F$ to a map from $D^l \times [0, 1]$, which is a homotopy from $f_n(W)^\Gamma \circ \lambda$ to $N_n(\psi)^\Gamma \circ \beta$, after possibly enlarging $W$ via some $\psi' \colon W \to W'$. Thus $f_n$ is a $\Sigma_n$-global equivalence.

Assume now that $f$ is an equivalence. For each compact Lie group $G$ and each $n \geqslant 0$, $\FMGn \subset \FNGn$ by Lemma~\ref{lemmcontain}. If $\Gamma \in \FNGn$ then there is a $G$-representation $W$ with $N_n(W)^\Gamma \neq \emptyset$. Since $f_n$ is a $\Sigma_n$-global equivalence a point in $N_n(W)^\Gamma$ (i. e. a map $\ast=D^0 \to N_n(W)^\Gamma$) can be lifted to $M_n(W')^\Gamma$ for some representation $W'$ containing $W$. Therefore $\FMGn= \FNGn$.
\end{proof}

\begin{prop}
\label{propproductsobspc}
Given two $\Nin$-symmetric sequences in $\Spc$ $M$ and $N$, their categorical product $M \times N$ is an $\Nin$-symmetric sequence with admissible graph subgroups $\FMGn \cap \FNGn$.
\end{prop}
\begin{proof}
Note that $\sobspc$ is a functor category, so limits are computed objectwise, and $(M \times N)_n = M_n \times N_n$. Since a product is non-empty if and only if each term is, and fixed points commute with products, $\FMNGn=\FMGn \cap \FNGn$.

If $\Gamma \in \FMGn \cap \FNGn$, consider any $G$-representation $V$, $l \geqslant 0$, and any map $(\alpha_0, \alpha_1) \colon \partial D^l \to (M \times N)_n(V)^\Gamma = M_n(V)^\Gamma \times N_n(V)^\Gamma$. Since $\Gamma \in \FMGn$, there is an embedding $\psi_0 \colon V \to W_0$ and a map $\lambda_0 \colon D^l \to M_n(W_0)^\Gamma$ such that $\lambda_0 \circ i_l = M_n(\psi_0)^\Gamma \circ \alpha_0$. And since $\Gamma \in \FNGn$, we can similarly find $\psi \colon W_0 \to W_1$ and $\lambda_1 \colon D^l \to N_n(W_1)^\Gamma$ such that $\lambda_1 \circ i_l = N_n(\psi_1)^\Gamma \circ N_n(\psi_0)^\Gamma \circ \alpha_1$. Then $(M_n(\psi_1)^\Gamma \circ \lambda_0,\lambda_1)$ is the map verifying that $M \times N$ is an $\Nin$-symmetric sequence. 
\end{proof}

Proposition~\ref{propproductsobspc} also implies that finite products of $\Nin$-symmetric sequences in $\Spc$ are $\Nin$, with sets of admissible graph subgroups given by the intersection of the respective sets of admissible graph subgroups of each factor. Later in Section~\ref{sectionexamples} we will construct examples of global $\Nin$-operads, and one important step will require taking arbitrary products of global $\Nin$-operads. However the statement of the previous proposition is not true for arbitrary products, since arbitrary products of global equivalences of orthogonal spaces are not necessarily global equivalences (see the comment before \cite[Theorem 1.1.10]{global}). We will prove that arbitrary products of \emph{fibrant} $\Nin$-symmetric sequences in $\Spc$ are $\Nin$, using the following result.

\begin{lemm}
\label{lemmfibrantss}
Let $M \in \sobspc$ be fibrant, in the sense of Remark~\ref{remmodelssfib}. Then $M$ is $\Nin$ if and only if for each compact Lie group $G$, each $n \geqslant 0$, each $\Gamma \in \FMGn$ and each faithful $G$-representation $V$ the space $M_n(V)^\Gamma$ is weakly contractible.
\end{lemm}
\begin{proof}
Assume that $M$ is $\Nin$. If $\Gamma \in \FMGn$ and $V$ is a faithful $G$-representation, let $\alpha$ denote a map $\partial D^l \to M_n(V)^\Gamma$. By the definition of $\Nin$ there is a map $\lambda$ that makes the triangular diagram in Definition~\ref{defiNin} commute after applying some $G$-embedding $\psi \colon V \to W$. Now since $M$ is fibrant, $M_n(\psi)^\Gamma$ is a weak homotopy equivalence, and $\lambda$ witnesses that $M_n(\psi)^\Gamma \circ \alpha$ is nullhomotopic. Therefore $\alpha$ itself is nullhomotopic, and $M_n(V)^\Gamma$ is weakly contractible.

We now assume that each $M_n(W)^\Gamma$ is weakly contractible for all faithful $G$-representations $W$. Then we can always find the lift $\lambda$ of Definition~\ref{defiNin} by taking $\psi$ to be an embedding into a faithful $G$-representation $W$.
\end{proof}

\begin{prop}
\label{proparbiproductsobspc}
Let $\{M^j\}_{j \in J}$ be a set of fibrant $\Nin$-symmetric sequences, then their categorical product $\prod_{j \in J} M^j$ is a fibrant $\Nin$-symmetric sequence with admissible graph subgroups $\bigcap_{j \in J} \FF^{M^j}_{G, n}$.
\end{prop}
\begin{proof}
Since a product is non-empty if and only if each term is, and fixed points commute with arbitrary products, the admissible graph subgroups of $\prod_{j \in J} M^j$ are precisely $\bigcap_{j \in J} \FF^{M^j}_{G, n}$.

Arbitrary products of fibrant symmetric sequences in $\Spc$ are fibrant by the lifting property that defines fibrations. If each $M^j_n(V)^\Gamma$ is weakly contractible, so is $\prod_{j \in J} M^j_n(V)^\Gamma$, and then Lemma~\ref{lemmfibrantss} implies the result.
\end{proof}

\section{Global \texorpdfstring{$\Nin$}{N-infinity}-operads and global transfer systems}
\label{sectionoperadtransfer}

In this section we define global $\Nin$-operads and global transfer systems, which are our main objects of study. In the previous section we studied the $\Nin$-symmetric sequences in $\Spc$, and just as in the $G$-equivariant setting, a global $\Nin$-operad is simply an operad whose underlying symmetric sequence is $\Nin$.

\begin{defi}
\label{defiNinoperad}
A \emph{global $\Nin$-operad} is an operad in $(\Spc, \boxtimes)$, orthogonal spaces with the box product, whose underlying symmetric sequence in $\Spc$ is $\Nin$ in the sense of Definition~\ref{defiNin}. We use $\Ningl$ to denote the category of global $\Nin$-operads.  An \emph{equivalence of global $\Nin$-operads} is a morphism of global $\Nin$-operads whose underlying morphism of symmetric sequences in $\Spc$ is an equivalence in the sense of Definition~\ref{defiequivnin}. An equivalence of global $\Nin$-operads induces a Quillen equivalence between the respective categories of algebras in orthogonal spaces by \cite[Theorem II]{globaloperads}.
\end{defi}

\begin{rem}
We do not require a global $\Nin$-operad to be $\Sigma_n$-free as this is not required to obtain a model structure on its category of algebras (by \cite[Theorem I]{globaloperads}). This is in contrast to the case of $G$-spaces, where $\Sigma_n$-freeness is necessary to obtain a model structure on algebras, and thus it is a part of the definition of equivariant $\Nin$-operad, see \cite[Definition~3.7]{BLUMBERG2015658}.
\end{rem}

Let $H$ be a compact Lie group, and let $A$ be a finite $H$-set. The $H$-set $A$ is characterized by a continuous homomorphism $\phi_A \colon H \to \Sigma_n$ where $n = \lvert A \rvert$. We call the graph subgroup of this homomorphism the \emph{graph subgroup associated to $A$}, and we denote it by $\Gamma_A \in \fat(H, \Sigma_n)$. To be more precise, we should say that picking an ordering of the elements of $A$ determines $\phi_A$, but any two orderings yield homomorphisms $\phi_A$ that are conjugate via an element of $\Sigma_n$.

We say that a finite $H$-set $A$ is admissible for $M$ an $\Nin$-symmetric sequence in $\Spc$ if $\Gamma_A$ is admissible for $M$, that is, if $\Gamma_A \in \FMHn$ for $n = \lvert A \rvert$. By Lemma~\ref{lemmFMGn} \ref{lemmFMGn4} $\FMHn$ is closed under conjugation, so the choice of the ordering mentioned in the previous paragraph does not affect the admissibility of $A$. The admissible graph subgroups of an $\Nin$-symmetric sequence completely determine its admissible sets and vice versa.

Now let $K$ be a closed finite index subgroup of $H$. We say that the abstract \emph{transfer from $K$ to $H$} is admissible for $M \in \sobspc$ if the orbit $H/K$ is admissible for $M$ as an $H$-set. Each finite $H$-set decomposes as a disjoint union of orbits, and we will see later that the operadic structure of a global $\Nin$-operad imposes that the admissible $H$-sets are closed under disjoint unions and sub-$H$-sets. This will imply that the admissible sets (and thus the admissible graph subgroups) of a global $\Nin$-operad are completely determined by its admissible transfers, and therefore we can completely classify global $\Nin$-operads up to equivalence through their admissible transfers.  This leads to the following definition.

\begin{defi}[Global transfer system]
\label{defitransfer}
 Let $\Lie$ denote the class of all compact Lie groups. A \emph{global transfer system} $T$ is a binary relation $\leqslant_T$ on $\Lie$ such that:
\begin{enumerate}[\roman*)]
    \item The relation $\leqslant_T$ is a partial order (that is, it is transitive, reflexive and antisymmetric) that refines the finite index subgroup relation on $\Lie$ (if $ K \leqslant_T H$ then $K$ is a closed finite index subgroup of $H$).
    \item For any $G,H, K \in \Lie$ and any continuous homomorpism $\theta \colon G \to H$, if $K \leqslant_T H$ then $\theta^{-1}(K) \leqslant_T G$.
\end{enumerate}
\end{defi}

\begin{rem}
The global transfer systems form a poset with respect to inclusion, which we denote by $\Ilie$. Explicitly, $T \subset T'$ for $T, T' \in \Ilie$ if $K \leqslant_T H$ implies that $K \leqslant_{T'} H$. Condition ii) implies that continuous isomorphisms of compact Lie groups preserve the relation $\leqslant_T$, and there are only countably many isomorphism classes of compact Lie groups, thus $\Ilie$ is indeed a set. Arbitrary intersections of global transfer systems are again global transfer systems, and these give the meets of the poset $\Ilie$. There is a greatest global transfer system $\All$, for which $K \leqslant_\All H$ for any closed finite index subgroup $K \leqslant H$. Therefore $\Ilie$ also has arbitrary joins, and is thus a complete lattice. Joins in $\Ilie$ are the transitive closures of unions.
\end{rem}

\begin{rem}
\label{remindextransfer}
In the $G$-equivariant case, $\Nin$-operads are completely classified by their associated \emph{$G$-indexing systems} of admissible $H$-sets for $H \leqslant G$, which were defined in \cite[Definition 3.22]{BLUMBERG2015658}. As noted in \cite[Section~3]{Rubin2} and \cite[Lemma 6]{balchin2021ninftyoperads}, this definition is equivalent to that of a \emph{$G$-transfer system} (\cite[Definition~3.4]{Rubin2}), which only keeps track of the admissible orbits. As this second definition is simpler, we have chosen the analogous Definition~\ref{defitransfer} for a global transfer system.
\end{rem}

\begin{rem}
A $G$-transfer system is closed under conjugation by elements of $G$ and restriction to a subgroup. However according to Condition ii) of Definition~\ref{defitransfer} a global transfer system has to be closed under pullback along arbitrary continuous homomorphisms. This corresponds directly to the extra structure that globally equivariant objects possess.
\end{rem}

\section{The structure of the homotopy category of global \texorpdfstring{$\Nin$}{N-infinity}-operads}

Our first goal in this section is to check that the admissible transfers of a global $\Nin$-operad assemble into a global transfer system. After that we will check that a morphism between two
global $\Nin$-operads is an equivalence if and only if their associated global transfer systems are equal. Finally we will prove that the homotopy category of global $\Nin$-operads with respect to their equivalences is thin and embeds fully into the poset of global transfer systems. We will check that this embedding is essentially surjective in the next section.

\begin{constr}
Let $\Op$ be a global $\Nin$-operad. We construct a relation $T_\Op$ on $\Lie$ associated to $\Op$. If $H$ is a compact Lie group and $K \leqslant H$ is a closed finite index subgroup, we say that $K \leqslant_{T_\Op} H$ if the orbit $H / K$ is admissible for $\Op$ as an $H$-set, i. e. if $\Gamma_{H/K} \in \FOpHn$ where $n = \lvert H/K \rvert$.
\end{constr}

As seen in Proposition~\ref{propequininss}, whether a morphism $f \colon \Op \to \Pop$ of global $\Nin$-operads is an equivalence depends only on the admissible graph subgroups of the two operads, $\FOpHn$ and $\FPopHn$ for varying $H$ and $n$. Each graph subgroup $\Gamma \in \fat(G, \Sigma_n)$ determines an $H$-set with $n$ elements for some subgroup $H \leqslant G$. In the following lemma we check that a disjoint union of $H$-sets is admissible for a global $\Nin$-operad if and only if each of the terms is admissible. This will imply that whether a morphism $f \colon \Op \to \Pop$ of global $\Nin$-operads is an equivalence depends only on the admissible orbits, and hence on the associated relations $T_\Op$ and $T_\Pop$.

\begin{lemm}
\label{lemmdisjointunion}
Let $\Op$ be a global $\Nin$-operad. Let $R$ and $S$ be finite $H$-sets, with $n = \lvert R \rvert$ and $m = \lvert S \rvert$. Then $\Gamma_{R \coprod S} \in \FOpHnm$ if and only if $\Gamma_R \in \FOpHn$ and $\Gamma_S \in \FOpHm$.
\end{lemm}
\begin{proof}
We adapt the arguments from \cite[Lemma 4.10]{BLUMBERG2015658} and \cite[Lemma 4.15]{BLUMBERG2015658}. Assume that $\Gamma_R \in \FOpHn$ and $\Gamma_S \in \FOpHm$. By Lemma~\ref{lemmdefsoperad} the operadic structure on $\Op$ determines a $(\Sigma_n \times \Sigma_m)$-equivariant composition morphism of orthogonal spaces
\[\gamma \colon \Op_2 \boxtimes \Op_n \boxtimes \Op_m \to \Op_{n + m}.\]

Since $\Op$ is an $\Nin$-symmetric sequence, there are $H$-representations $V_0$, $V_1$ and $V_2$ such that $\Op_n(V_1)^{\Gamma_R} \neq \emptyset$, $\Op_m(V_2)^{\Gamma_S} \neq \emptyset$, and since graphs of trivial homomorphisms are admissible by Lemma~\ref{lemmFMGn} \ref{lemmFMGn5}, $\Op_2(V_0)^{H}\neq\emptyset$. Let $\Gamma_{R \coprod S} \leqslant H \times \Sigma_n \times \Sigma_m \leqslant H \times \Sigma_{n + m}$ denote the graph subgroup that represents the $H$-set $R \coprod S$. The projections of $\Gamma_{R \coprod S}$ to $H \times \Sigma_n$ and $H \times \Sigma_m$ are $\Gamma_R$ and $\Gamma_S$ respectively.

Since $\gamma$ is $(\Sigma_n \times \Sigma_m)$-equivariant the universal property of the box product of orthogonal spaces (see \cite[Section 1.3]{global}) yields an $(H \times \Sigma_n \times \Sigma_m)$-equivariant map
\[\Op_2(V_0) \times \Op_n(V_1) \times \Op_m(V_2) \to \Op_{n + m}(V_0 \oplus V_1 \oplus V_2).\]
The $\Gamma_{R \coprod S}$-fixed points of the source are
\[\Op_2(V_0)^{H} \times \Op_n(V_1)^{\Gamma_R} \times \Op_m(V_2)^{\Gamma_S} \neq \emptyset,\]
so $\Op_{n + m}(V_0 \oplus V_1 \oplus V_2)^{\Gamma_{R \coprod S}}$ is also non-empty, and therefore $\Gamma_{R \coprod S} \in \FOpHnm$.

Assume now that $\Gamma_{R \coprod S} \in \FOpHnm$. By Lemma~\ref{lemmdefsoperad} we obtain a $(\Sigma_n \times \Sigma_m)$-equivariant morphism
\[\Op_{n + m} \boxtimes \Op_1^{\boxtimes n} \boxtimes \Op_0^{\boxtimes m} \to \Op_n\]
where $\Sigma_m$ acts trivially on $\Op_n$. Since $\Op$ is an $\Nin$-symmetric sequence, there are $H$-representations $V_0$, $V_1$ and $V_2$ such that $\Op_{n + m}(V_0)^{\Gamma_{R \coprod S}} \neq \emptyset$, $\Op_1(V_1)^{H} \neq \emptyset$ and $\Op_0(V_2)^{H} \neq \emptyset$. We obtain an $(H \times \Sigma_n \times \Sigma_m)$-equivariant map 
\begin{equation}
\label{propopgivesindexing1}
\Op_{n + m}(V_0) \times \Op_1(V_1)^{\times n} \times \Op_0(V_2)^{\times m} \to \Op_n(V_0 \oplus V_1^{\oplus n} \oplus V_2^{\oplus m}).
\end{equation}
where on the target $H$ acts via the $H$-representation $V_0 \oplus V_1^{\oplus n} \oplus V_2^{\oplus m}$, the group $\Sigma_n$ acts via the action on $\Op_n$ and by permuting the terms of $V_1^{\oplus n}$, and $\Sigma_m$ acts by permuting the terms of $V_2^{\oplus m}$. The permutations on $V_1^{\oplus n}$ and $V_2^{\oplus m}$ arise from the fact that permuting the terms of $\Op_1(V_1)^{\times n}$ and $\Op_0(V_2)^{\times m}$ also permutes the ordering of the copies of $V_1$ and $V_2$.

By the previous paragraph the $\Gamma_{R \coprod S}$-fixed points of the source of (\ref{propopgivesindexing1}) are non-empty, so $\Op_n(V_0 \oplus V_1^{\oplus n} \oplus V_2^{\oplus m})^{\Gamma_{R \coprod S}} \neq \emptyset$.

Let $W$ be the $H$-representation whose underlying linear isometric space is $V_0 \oplus V_1^{\oplus n} \oplus V_2^{\oplus m}$, but with $H$-action given by pulling back the $H \times \Sigma_n \times \Sigma_m$-action along the homomorphism $H \to H \times H \to H \times \Sigma_n \times \Sigma_m$ given by $\Gamma_{R \coprod S}$. Then 
\[\emptyset \neq \Op_n(V_0 \oplus V_1^{\oplus n} \oplus V_2^{\oplus m})^{\Gamma_{R \coprod S}} = \Op_n(W)^{\Gamma_{R \coprod S}}=\Op_n(W)^{\Gamma_R}\]
and so $\Gamma_R \in \FOpHn$. By exchanging the roles of $R$ and $S$ in the previous argument we can prove that $\Gamma_S \in \FOpHm$.
\end{proof}

We now check that $T_\Op$ is actually a global transfer system. For this the operadic composition morphisms will again be essential, in order to check that $T_\Op$ is transitive.

\begin{prop}
\label{propopgivesindexing}
Let $\Op$ be a global $\Nin$-operad. The relation $T_\Op$ is a global transfer system.
\end{prop}
\begin{proof}
By construction $T_\Op$ refines the closed finite index subgroup relation, and is thus antisymmetric. By Lemma~\ref{lemmFMGn} \ref{lemmFMGn5} the trivial $H$-set $H/H$ is admissible for $\Op$ for any $H$, so $T_\Op$ is reflexive.

Let G, H and K be compact Lie groups with $K \leqslant H$ a subgroup of index $n$, and let $\theta \colon G \to H$ be a continuous homomorphism. Assume that $K \leqslant_{T_\Op} H$, that is, $\Gamma_{H/K} \in \FOpHn$. By Lemma~\ref{lemmFMGn} \ref{lemmFMGn2} the $G$-set $\theta^\ast(H/K)$ is also admissible for $\Op$. The $G$-stabilizer of the equivalence class of the unit of $H$ in $\theta^\ast(H/K)$ is precisely $\theta^{-1}(K)$, so $\theta^\ast(H/K)$ contains the orbit $G / \theta^{-1}(K)$. By Lemma~\ref{lemmdisjointunion} the orbit $G / \theta^{-1}(K)$ is admissible for $\Op$, and thus $\theta^{-1}(K) \leqslant_{T_\Op} G$, showing that condition ii) in the definition of global transfer system holds.

Lastly we check the transitivity of $T_\Op$. For this we adapt the argument from \cite[Lemma 4.12]{BLUMBERG2015658}. Let $K \leqslant H \leqslant G$, with $K$ of index $n$ in $H$ and $H$ of index $m$ in $G$. Assume that $K \leqslant_{T_\Op} H$ and $H \leqslant_{T_\Op} G$.

The $G$-action on $G / H$ determines a homomorphism $\phi_{G / H} \colon G \to \Sigma_m$, and the $H$-action on $H /K$ determines a homomorphism $\phi_{H / K} \colon H \to \Sigma_n$. Pick $(g_1, \dots, g_m)$ a set of coset representatives of $G / H$. This determines a homomorphism $\beta \colon G \to \Sigma_m \wr H$. The composite
\[G \to \Sigma_m \wr H \to \Sigma_m \wr \Sigma_n \to \Sigma_{m n}\]
given by $\beta$ and $\phi_{H / K}$ is precisely the homomorphism that represents the $G$-action on $G / K$. This also implies that $\Gamma_{G/K} \leqslant G \times \Sigma_{m n}$ is contained in $G \times (\Sigma_m \wr \Sigma_n)$.

By Lemma~\ref{lemmdefsoperad} we have a $\Sigma_m \wr \Sigma_n$-equivariant morphism
\[\Op_m \boxtimes \Op_n^{\boxtimes m} \to \Op_{m n}.\]
By admissibility, we can pick a $G$-representation $V_0$ such that $\Op_m(V_0)^{\Gamma_{G / H}} \neq \emptyset$. Similarly we can choose an $H$-representation $V_1$ such that $\Op_n(V_1)^{\Gamma_{H/K}} \neq \emptyset$. By \cite[III Theorem 4.5]{diecklie} $V_1$ embeds $H$-equivariantly in some $G$-representation $V_1'$ and so there exists some point $x$ in $\Op_n(V_1')^{\Gamma_{H/K}}$. Then we obtain a $G \times (\Sigma_m \wr \Sigma_n)$-equivariant map 
\begin{equation}
\label{propopgivesindexing2}
\Op_m(V_0) \times \Op_n(V_1')^{\times m} \to \Op_{m n}(V_0 \oplus V_1'^{\oplus m}).
\end{equation}
The point $(g_1 x, \dots , g_m x) \in \Op_n(V_1')^{\times m}$ is fixed by the action of $\Gamma_{G/K} \leqslant G \times \Sigma_m \wr \Sigma_n$. An explicit check of this fact can be found in the proof of \cite[Lemma 4.12]{BLUMBERG2015658}. This point depends on our choice of representatives $(g_1, \dots, g_m)$, but we only care about finding a $\Gamma_{G/K}$-fixed point in $\Op_n(V_1')^{\times m}$.

Since $\Sigma_n$ acts trivially on $\Op_m(V_0)$, $\Op_m(V_0)^{\Gamma_{G / K}} = \Op_m(V_0)^{\Gamma_{G / H}} \neq \emptyset$, and the $\Gamma_{G/K}$-fixed points of the source of (\ref{propopgivesindexing2}) are non-empty, thus the same is true for the target of (\ref{propopgivesindexing2}). As before, let $W$ be the $G$-representation $V_0 \oplus V_1^{\oplus m}$ with the $G$-action given by pulling back the $G \times (\Sigma_m \wr \Sigma_n)$-action along $\phi_{G / H}$ (note that on $V_0 \oplus V_1^{\oplus m}$ the copies of $\Sigma_n$ act trivially). Thus $\Op_{m n}(W)^{\Gamma_{G/K}}$ is non-empty, and so $K \leqslant_{T_\Op} G$.
\end{proof}

\begin{prop}
\label{propequiNinoperad}
Let $f \colon \Op \to \Pop$ be a morphism of global $\Nin$-operads. Then $f$ is an equivalence if and only if  $T_\Op = T_{\Pop}$.
\end{prop}
\begin{proof}
By Proposition~\ref{propequininss} the morphism $f$ is an equivalence if and only if $\Op$ and $\Pop$ have the same admissible graph subgroups, or equivalently, the same admissible sets. Each finite $H$-set decomposes as a disjoint union of orbits. Lemma~\ref{lemmdisjointunion} implies that a finite $H$-set is admissible for a global $\Nin$-operad if and only if each of the $H$-orbits that compose it is admissible for the operad. And by construction the operads $\Op$ and $\Pop$ have the same admissible orbits precisely if $T_\Op = T_{\Pop}$.
\end{proof}

\begin{constr}
\label{constrT}
Using these results we construct a functor $T_{(-)}$ from the category of global $\Nin$-operads $\Ningl$ to $\Igl$ the poset of global transfer systems, ordered by inclusion. It sends a global $\Nin$-operad $\Op$ to its associated global transfer system $T_\Op$. It is indeed a functor by Lemma~\ref{lemmcontain}. By Proposition~\ref{propequiNinoperad} this functor sends equivalences of global $\Nin$-operads to identities. Therefore it descends to a functor 
\[\mathpzc{T} \colon \Ho(\Ningl) \to \Igl\]
from the homotopy category of global $\Nin$-operads with respect to their equivalences to the poset $\Igl$. 
\end{constr}

In Theorem~\ref{thmmain} we will see that this functor is essentially surjective. We check now that it is fully faithful, so that we eventually obtain that it is an equivalence of categories. We do this using the following general result.

\begin{lemm}
\label{lemmgeneralhoposet}
Let $(\cat, W)$ be a category with weak equivalences, that is, such that $W$ is a subcategory of $\cat$ that contains all isomorphisms and satisfies the 2-out-of-3 property. Assume that $\cat$ has all finite products, and that there is a poset $P$ and a functor $F \colon \cat \to P$ that sends weak equivalences to identities. Furthermore, assume that if $F(X) \leqslant F(Y)$ then the projection $\pi_X \colon X \times Y \to X$ is a weak equivalence. Then $\Ho(\cat)$ the homotopy category of $(\cat, W)$ (its localization at $W$) is a thin category (or preordered class), and the functor $\Ho(\cat) \to P$ induced by $F$ is fully faithful.
\end{lemm}
\begin{proof}
The homotopy category of $\cat$ is characterized by a functor $Q \colon \cat \to \Ho(\cat)$ that sends morphisms in $W$ to isomorphisms and satisfies a certain universal property. We first prove that under our assumptions, for any two morphisms $f, g \colon X \to Y$, in the homotopy category $Q(f)= Q(g)$.

Consider the morphisms $(\id_X, f), (\id_X, g) \colon X \to X \times Y$. When postcomposed with $\pi_X \colon X \times Y \to X$ they become equal. Therefore $Q(\pi_X) \circ Q((\id_X, f)) = Q(\pi_X) \circ Q((\id_X, g))$. The existence of morphisms from $X$ to $Y$ shows that $F(X) \leqslant F(Y)$, and then by our assumptions $Q(\pi_X)$ is an isomorphism, so $Q((\id_X, f)) = Q((\id_X, g))$. This implies that $Q(f) = Q(\pi_Y) \circ Q((\id_X, f)) = Q(\pi_Y) \circ Q((\id_X, g)) = Q(g)$.

We will use the general facts that $Q$ is a bijection on objects, and each morphism of the homotopy category $\Ho(\cat)$ is represented by a composition of morphisms which are either of the form $Q(f)$ for $f$ a morphism in $\cat$, or $Q(g)^{-1}$ for $g \in W$ a weak equivalence in $\cat$.

We now prove that for any two $X, Y \in \cat$, $\Ho(\cat)(X, Y)$ is non-empty if and only if $F(X) \leqslant F(Y)$. First note that if $F(X) \leqslant F(Y)$ then $Q(\pi_Y) \circ Q(\pi_X)^{-1}$ is a morphism in $\Ho(\cat)$ from $X$ to $Y$. The functor $F$ sends weak equivalences of $\cat$ to identities, so by the universal property of the homotopy category it induces a functor $\mathpzc{F} \colon \Ho(\cat) \to P$ with $\mathpzc{F} \circ Q = F$ that shows that if $\Ho(\cat)(X, Y)$ is non-empty then $F(X) \leqslant F(Y)$.

Now we check that if $\Ho(\cat)(X, Y)$ is non-empty, any $\alpha \in \Ho(\cat)(X, Y)$ equals $Q(\pi_Y) \circ Q(\pi_X)^{-1}$ using induction on the length of a zigzag that composes to $\alpha$ (a composition in $\Ho(\cat)$ of morphisms of the form $Q(f)$ or $Q(g)^{-1}$ for $g$ a weak equivalence in $\cat$). For the base case, assume that $\alpha=Q(f)$ for $f \colon X \to Y$ a morphism in $\cat$ or $\alpha = Q(g)^{-1}$ for $g \colon Y \to X$ a weak equivalence in $\cat$. In the first case, by the statement that we proved in the beginning, $Q(f \circ \pi_X) = Q(\pi_Y)$ and thus $\alpha = Q(f) = Q(\pi_Y) \circ Q(\pi_X)^{-1}$. In the second case $Q(g \circ \pi_Y) = Q(\pi_X)$ and thus $\alpha = Q(g)^{-1} = Q(\pi_Y) \circ Q(\pi_X)^{-1}$.

For the induction step, let $\alpha \in \Ho(\cat)(X, Y)$ be represented by a zigzag of length at least $2$. Split it into a zigzag from $X$ to some $X'$ and another zigzag from $X'$ to $Y$, each of length at least $1$. We switch now to using the notation $\pi^{X \times Y}_X \colon X \times Y \to X$ for the projection, to avoid ambiguity. By the induction hypothesis, the compositions in $\Ho(\cat)$ of each of these two smaller zigzags are $Q(\pi^{X \times X'}_{X'}) \circ Q(\pi^{X \times X'}_X)^{-1}$ and $Q(\pi^{X' \times Y}_Y) \circ Q(\pi^{X' \times Y}_{X'})^{-1}$ respectively. Consider the following commutative diagram in $\cat$
\[\begin{tikzcd}
	&& {X \times Y} \\
	&& {X \times X' \times Y} & {} \\
	X & {X \times X'} & {X'} & {X' \times Y} & Y
	\arrow["{\pi^{X \times X'}_X}", from=3-2, to=3-1]
	\arrow["{\pi^{X \times X'}_{X'}}"', from=3-2, to=3-3]
	\arrow["{\pi^{X' \times Y}_{X'}}", from=3-4, to=3-3]
	\arrow["{\pi^{X' \times Y}_Y}"', from=3-4, to=3-5]
	\arrow["{\pi_{X, X'}}"', from=2-3, to=3-2]
	\arrow["{\pi_{X', Y}}", from=2-3, to=3-4]
	\arrow["{\pi_{X, Y}}", from=2-3, to=1-3]
	\arrow["{\pi^{X \times Y}_X}"', curve={height=18pt}, from=1-3, to=3-1]
	\arrow["{\pi^{X \times Y}_Y}", curve={height=-18pt}, from=1-3, to=3-5]
	\arrow["\simeq", draw=none, from=2-3, to=3-2]
	\arrow["\simeq"', draw=none, from=2-3, to=1-3]
	\arrow["\simeq"', draw=none, from=3-2, to=3-1]
	\arrow["\simeq", curve={height=18pt}, draw=none, from=1-3, to=3-1]
	\arrow["\simeq"', draw=none, from=3-4, to=3-3]
\end{tikzcd}\]
We apply $Q$ to this diagram, and then we can use the resulting diagram in $\Ho(\cat)$ to check that $\alpha = Q(\pi^{X' \times Y}_Y) \circ Q(\pi^{X' \times Y}_{X'})^{-1} \circ Q(\pi^{X \times X'}_{X'}) \circ Q(\pi^{X \times X'}_X)^{-1} = Q(\pi^{X \times Y}_Y) \circ Q(\pi^{X \times Y}_X)^{-1}$ in $\Ho(\cat)$.

We have proven that any set $\Ho(\cat)(X, Y)$ has either zero or one element, and thus $\Ho(\cat)$ is a thin category. Additionally, since $\Ho(\cat)(X, Y)$ is non-empty if and only if $F(X) \leqslant F(Y)$, the functor $\mathpzc{F} \colon \Ho(\cat) \to P$ induced by $F$ is fully faithful.
\end{proof}

\begin{prop}
\label{prophoglposet}
The homotopy category $\Ho(\Ningl)$ of global $\Nin$-operads is a thin category and the functor $\mathpzc{T} \colon \Ho(\Ningl) \to \Igl$ is fully faithful.
\end{prop}
\begin{proof}
The category of operads in $(\Spc, \boxtimes)$ is complete and the forgetful functor to symmetric sequences in $\Spc$ creates all limits. The equivalences of global $\Nin$-operads satisfy the 2-out-of-3 property because the equivalences of symmetric sequences in $\Spc$ do. Consider two global $\Nin$-operads $\Op$ and $\Pop$ such that $T_\Op \leqslant T_{\Pop}$. By Lemma~\ref{propproductsobspc} $\Op \times \Pop$ is a global $\Nin$-operad with admissible sets $T_\Op \cap T_{\Pop} = T_\Op$, so by Proposition~\ref{propequiNinoperad} the projection $\pi_\Op \colon \Op \times \Pop \to \Op$ is an equivalence. Then by Lemma~\ref{lemmgeneralhoposet} the homotopy category $\Ho(\Ningl)$ is thin and the functor $\mathpzc{T} \colon \Ho(\Ningl) \to \Igl$ is fully faithful.
\end{proof}

\begin{rem}
In the context of $G$-equivariant $\Nin$-operads, Blumberg and Hill prove in \cite[Proposition 5.5]{BLUMBERG2015658} a stronger result than Proposition~\ref{prophoglposet}. They check that the derived mapping space from any $G$-operad to a $G$-equivariant $\Nin$-operad is either weakly contractible or empty, while in the globally equivariant setting Proposition~\ref{prophoglposet} works only on the level of the homotopy category of global $\Nin$-operads.

We expect that the stronger statement about global $\Nin$-operads analogous to \cite[Proposition 5.5]{BLUMBERG2015658} is also true, but proving it might require more technical machinery.
\end{rem}

\section{Constructing global \texorpdfstring{$\Nin$-operads}{N-infinity}}
\label{sectionexamples}
In this section we prove our main result, Theorem~\ref{thmmain}, that states that the functor $\mathpzc{T} \colon \Ho(\Ningl) \to \Igl$ defined in Construction~\ref{constrT} is an equivalence between the homotopy category of $\Nin$-operads and the poset of global transfer systems. To do that, we will use Proposition~\ref{propGtoglobalnin} to construct examples of global $\Nin$-operads starting from $\Nin$-operads for a single group. We start with some technical results.

First of all, we need to relate orthogonal spaces and $G$-spaces. For each compact Lie group $G$ we have the \emph{underlying $G$-space} functor $(-)(\Ug) \colon \Spc \to \GTopcat$ given by
\[X(\Ug) = \colim_{V \in s(\Ug)} X(V)\]
where $s(\Ug)$ is the poset of finite-dimensional subrepresentations of our chosen complete $G$-universe $\Ug$. This functor has a right adjoint $R_G \colon \GTopcat \to \Spc$ given by
\[R_G(Y)(V) = \Map^G(\Lcat(V, \Ug), Y)\]
for each $V \in \Lcat$. Here $\Lcat(V, \Ug)$ denotes the space of linear isometric embeddings from $V$ to $\Ug$, with $G$-action induced by the $G$-action on $\Ug$. For details on this construction see \cite[Construction~1.2.25]{global}.

\begin{lemm}
\label{lemmURstrong}
The functors $(-)(\Ug)$ and $R_G$ are strong symmetric monoidal with respect to the symmetric monoidal structures given by the categorical products in $\Spc$ and $\GTopcat$.
\end{lemm}
\begin{proof}
The functor $R_G$ is strong symmetric monoidal because it is a right adjoint, so it preserves products.

The functor $(-)(\Ug)$ can be realized as a sequential colimit indexed by a nested sequence of finite-dimensional subrepresentations of $\Ug$ that cover $\Ug$. Limits in $\Spc$ are computed in $\Topcat$, and in $\Topcat$ finite products commute with sequential colimits, so $(-)(\Ug)$ is strong symmetric monoidal.
\end{proof}

Note that the functor $(-)(\Ug)$ is not lax monoidal if considered with the monoidal structure given by the box product. It is however oplax symmetric monoidal.

\begin{rem}
\label{remNintimes}
By Lemma~\ref{lemmURstrong}, for each $G$ the underlying $G$-operad of an operad in $(\Spc, \times)$ (with respect to the categorical product on $\Spc$) is an operad in $G$-spaces. If an operad $\Op$ in $(\Spc, \times)$ is $\Sigma_n$-free and its underlying symmetric sequence in $\Spc$ is $\Nin$ and closed, then its underlying $G$-operad is an $\Nin$-operad in the sense of \cite{BLUMBERG2015658} (see Remark~\ref{remNinunderlying}). Note that Blumberg and Hill defined $\Nin$-operads for a finite group $G$, and Gutiérrez and White extended the definition to include compact Lie groups. Since the functor $(-)(\Ug)$ is not lax monoidal from $(\Spc, \boxtimes)$ to $(\GTopcat, \times)$, the underlying $G$-spaces of an operad in $(\Spc, \boxtimes)$ (with respect to the box product) do \emph{not} generally form an operad.
\end{rem}

Let $G$ and $H$ be compact Lie groups, we use $\GHTopcat$ to denote the category of $(G \times H)$-spaces. The model structure on $\GHTopcat$ that one usually works with in equivariant homotopy theory is the "genuine" model structure. In it, weak equivalences are constructed with respect to all subgroups of $G \times H$. But there are more model structures on $\GHTopcat$ that one can consider. In fact, for each set $F$ of closed subgroups of $G \times H$ there is a model structure on $\GHTopcat$ where the weak equivalences (respectively the fibrations) are the $(G \times H)$-equivariant maps such that restricting to the fixed points of any subgroup in $F$ yields a weak homotopy equivalence (respectively a Serre fibration). For example if one takes $F$ to be the set containing only the trivial subgroup, one obtains the "naive" equivariant model structure. An explicit construction of these model structures can be found in \cite[Proposition B.7]{global}.

In this section we will consider the model structure on $\GHTopcat$ given by the graph subgroups, the one obtained by taking $F=\fat(G, H)$. We call this model structure the $\fat(G, H)$-model structure.

We now look at the bifunctor of "$(G \times H)$-equivariant maps"
\begin{equation}
\label{eqmap}
    \Map^{G \times H}(-, -) \colon \GHTopcat^\opp \times \GHTopcat \to \Topcat.
\end{equation}
Together with the bifunctor
\begin{equation}
\label{eqtimes}
    - \times - \colon \GHTopcat \times \Topcat \to \GHTopcat
\end{equation}
it forms an \emph{adjunction of two variables}, see for example \cite[Definition 4.1.12]{hovey2007model} for the definition. 

\begin{lemm}
\label{lemmadjtwovariables}
The adjunction given by (\ref{eqmap}) and (\ref{eqtimes}) is a \emph{Quillen adjunction of two variables}, as defined in \cite[Definition 4.2.1]{hovey2007model}, for the $\fat(G, H)$-model structure in $\GHTopcat$, and the Quillen-Serre model structure on $\Topcat$.
\end{lemm}
\begin{proof}
By definition we have to check that given cofibrations $f \colon A \to B$ in $\GHTopcat$ and $g \colon X \to Y$ in $\Topcat$, their pushout product with respect to the functor $- \times -$,
\[f \square g \colon B \times X \coprod_{A \times X} A \times Y \to B \times Y\]
is a cofibration in $\GHTopcat$, and is acyclic if any of $f$ or $g$ are acyclic. By~\cite[Lemma 4.2.4]{hovey2007model}, it is enough to check this on the generating (acyclic) cofibrations.

Generating (acyclic) cofibrations of $\Topcat$ are the inclusions $i_l \colon \partial D^l \to D^l$ (respectively $j_l \colon D^l \to D^l \times [0, 1]$) for $l \geqslant 0$. A generating (acyclic) cofibration of the $\fat(G, H)$-model structure has the form $(G \times H)/\Delta \times i_l$ (respectively $(G \times H)/\Delta \times j_l$) with $\Delta \in \fat(G, H)$ a graph subgroup of $G \times H$. The pushout product $i_l \square i_{l'}$ is homeomorphic to $i_{l + l'}$, and similarly $i_l \square j_{l'}$ is homeomorphic to $j_{l + l'}$. Therefore in this case the pushout product of two generating cofibrations is a generating cofibration in the $\fat(G, H)$-model structure, and the pushout product of a generating cofibration and a generating acyclic cofibration is a generating acyclic cofibration, so the adjunction is a Quillen adjunction of two variables.
\end{proof}

\begin{lemm}
\label{lemmLgraphcof}
Let $G$ and $H$ be compact Lie groups, and let $W$ be a faithful $H$-representation. The $(G \times H)$-space $\Lcat(W, \Ug)$ is cofibrant in the $\fat(G, H)$-model structure.
\end{lemm}
\begin{proof}
By \cite[Proposition 1.1.19 ii)]{global} the space $\Lcat(W, \Ug)$ is cofibrant in the genuine model structure on $\GHTopcat$  given by all subgroups. This means that it is a retract of a generalized $(G \times H)$-CW-complex, built out of cells of the form $(G \times H)/\Delta \times D^l$, for subgroups $\Delta \leqslant G \times H$ and $l \geqslant 0$.

Since $W$ is a faithful $H$-representation, the $H$-action on $\Lcat(W, \Ug)$ is free. For each cell $(G \times H)/\Delta \times D^l$ that appears in the decomposition of $\Lcat(W, \Ug)$, there is a $(G \times H)$-equivariant map
\[f \colon (G \times H)/\Delta \times D^l \to \Lcat(W, \Ug).\]
Since $\Lcat(W, \Ug)$ is $H$-free, this means that $\Delta$ is a graph subgroup of $G \times H$. Indeed, if $(e_G, h) \in \Delta$, then $(e_G, e_H) \times \{0\} \in (G \times H)/\Delta \times D^l$ is fixed by $(e_G, h)$. Thus $f((e_G, e_H) \times \{0\}) \in \Lcat(W, \Ug)$ is also fixed by $(e_G, h)$, which means that $h = e_H$. Since only graph subgroups $\Delta \in \fat(G, H)$ can appear in the cell decomposition of $\Lcat(W, \Ug)$, it is also cofibrant in the $\fat(G, H)$-model structure.
\end{proof}

The following proposition uses the functor $R_G$ introduced at the beginning of this section to produce  examples of global $\Nin$-operads and $\Nin$-symmetric sequences in $\Spc$.

\begin{prop}
\label{propGtoglobalnin}
Consider a compact Lie group $G$, and an $\Nin$-symmetric sequence $\{\Opn\}_{n \in \NN}$ in $\GTopcat$ with associated family of admissible graph subgroups $\FOpGn$ for each $n \geqslant 0$ (in the sense of \cite[Definition 4.1]{BLUMBERG2015658}, i. e. such that each $\Op_n$ is $\Sigma_n$-free and a universal space for the family $\FOpGn$). Then $R_G(\Op) \defeq \{R_G(\Opn)\}_{n \in \NN}$ is a fibrant $\Nin$-symmetric sequence in $\Spc$ with associated sets of admissible graph subgroups
\[\FRGOpKn = \{\Gamma_\phi \in \fat(K, \Sigma_n) \; \text{with} \; \phi \colon H \to \Sigma_n \mid \forall \theta \colon  L \to H \; \text{with} \; L \leqslant G , \Gamma_{\phi \circ \theta} \in \FOpGn\}.\]
\end{prop}

Note that $\Sigma_n$-freeness is part of Blumberg and Hill's definition of $\Nin$-symmetric sequence, but it is not actually required in the proof of this proposition.

\begin{proof}
For each $W \in \Lcat$, the $\Sigma_n$-orthogonal space $R_G(\Opn)$ evaluated at $W$ is given by
\[R_G(\Opn)(W) = \Map^G(\Lcat(W, \Ug), \Opn).\]

Here $\Ug$ is a chosen complete $G$-universe, and we are considering the space of $G$-equivariant maps from $\Lcat(W, \Ug)$ to $\Opn$. The space $R_G(\Opn)(W)$ is always non-empty, as according to \cite[Definition 4.1]{BLUMBERG2015658} the space $\Op_n$ always has non-empty $G$-fixed points. The $\Sigma_n$-action on $\Opn$ induces a $\Sigma_n$-action on $R_G(\Opn)(W)$. A linear isometric embedding $\psi \colon W \to W'$ induces a map $\Lcat(W', \Ug) \to \Lcat(W, \Ug)$ by precomposition, and therefore a map $R_G(\Opn)(W) \to R_G(\Opn)(W')$, which gives $R_G(\Opn)$ the orthogonal space structure.

Now let $K$ be a compact Lie group, let $W$ be a $K$-representation, and let $n \geqslant 0$. By the previous paragraph the left $K$-action on $R_G(\Opn)(W)$ comes from the right $K$-action on $\Lcat(W, \Ug)$. Then $R_G(\Opn)(W)$ is a $(K \times \Sigma_n)$-space. Consider a graph subgroup $\Gamma_\phi \in \fat(K, \Sigma_n)$ associated to some homomorphism $\phi \colon H \to \Sigma_n$ with $H \leqslant K$. We need to compute the $\Gamma_\phi$-fixed points $R_G(\Opn)(W)^{\Gamma_\phi}$.

We consider $W$ as an $H$-representation by restricting the $K$-action. Then there is a homeomorphism
\[(\Map^G(\Lcat(W, \Ug), \Opn))^{\Gamma_\phi} \cong \Map^{G \times H}(\Lcat(W, \Ug), \phi^\ast(\Opn))\]
where $\phi^\ast(\Opn)$ is the $(G \times H)$-space obtained from pulling back the $\Sigma_n$-action through $\phi$.

We first check that the sets of admissible graph subgroups of $R_G(\Op)$ are indeed the $\FRGOpKn$ given in the statement of the theorem, i. e. that $\Gamma_\phi \in \fat(K, \Sigma_n)$ belongs to $\FRGOpKn$ if and only if there exists a $K$-representation $V$ such that $R_G(\Opn)(V)^\Gamma$ is non-empty. Assume that a graph subgroup $\Gamma_\phi \in \fat(K, \Sigma_n)$ with $\phi \colon H \to \Sigma_n$ does not belong to $\FRGOpKn$. This means that there is some subgroup $L \leqslant G$ and homomorphism $\theta \colon L \to H$ such that $\Gamma_{\phi \circ \theta} \notin \FOpGn$. Let $W$ be any $K$-representation. Pull back the restricted $H$-action on $W$ through $\theta$ to get an $L$-representation $\theta^\ast W$. By \cite[III Theorem 4.5]{diecklie} $\theta^\ast W$ embeds $L$-equivariantly into some $G$-representation, therefore it also embeds $L$-equivariantly into $\Ug$ via some $\psi \colon \theta^\ast W \to \Ug$. This means that $\psi \in \Lcat(W, \Ug)$ is fixed by $\Gamma_\theta \leqslant G \times H$. We will check by contradiction that
\[R_G(\Opn)(W)^{\Gamma_\phi} \cong \Map^{G \times H}(\Lcat(W, \Ug), \phi^\ast(\Opn))\]
is empty. Assume that there is some $(G \times H)$-equivariant map $f \colon \Lcat(W, \Ug) \to \phi^\ast(\Opn)$. Then $f(\psi) \in \phi^\ast(\Opn)$ is also fixed by $\Gamma_\theta$. However
\[(\phi^\ast(\Opn))^{\Gamma_\theta} = \Opn^{\Gamma_{\phi \circ \theta}} = \emptyset\]
because $\Gamma_{\phi \circ \theta} \notin \FOpGn$, giving a contradiction. Thus $R_G(\Opn)(W)^{\Gamma_\phi}$ is empty for any $K$-representation $W$.

To complete the proof we need to check that $R_G(\Opn)(W)^{\Gamma_\phi}$ is weakly contractible for any $\Gamma_\phi \in \FRGOpKn$ and any faithful $K$-representation $W$. Let $\Gamma_\phi$ be the graph of $\phi \colon H \to \Sigma_n$ with $H \leqslant K$. Then $W$ is also faithful as an $H$-representation. Thus by Lemma~\ref{lemmLgraphcof} the $(G \times H)$-space $\Lcat(W, \Ug)$ is cofibrant in the $\fat(G, H)$-model structure.

We want to compute
\[R_G(\Opn)(W)^{\Gamma_\phi} \cong \Map^{G \times H}(\Lcat(W, \Ug), \phi^\ast(\Opn)).\]
For any graph subgroup $\Gamma_\theta \in \fat(G, H)$,
\[(\phi^\ast(\Opn))^{\Gamma_\theta} = \Opn^{\Gamma_{\phi \circ \theta}} \simeq \ast\]
because $\Gamma_{\phi \circ \theta} \in \FOpGn$ since $\Gamma_\phi \in \FRGOpKn$. This means that $\phi^\ast(\Opn) \to \ast$ is an acyclic fibration in the $\fat(G, H)$-model structure. By~\cite[Lemma 4.2.2]{hovey2007model} since $\Lcat(W, \Ug)$ is cofibrant in the $\fat(G, H)$-model structure, $\Map^{G \times H}(\Lcat(W, \Ug), \phi^\ast(\Opn)) \to \ast$ is an acyclic Serre fibration. So $R_G(\Opn)(W)^{\Gamma_\phi}$ is weakly contractible. In particular it is non-empty, so if $\Gamma_\phi \in \FRGOpKn$ then picking any faithful $K$-representation $W$ shows that $\Gamma_\phi$ is admissible for $R_G(\Op)$. This completes the check that the $\FRGOpKn$ of the statement of the proposition are the sets of admissible graph subgroups of $R_G(\Op)$.

We check that $R_G(\Op)$ satisfies Definition~\ref{defiNin} and is thus a global $\Nin$-symmetric sequence in $\Spc$. Given any $K$-representation $V$, it embeds into a faithful $K$-representation $W$ via some $\psi \colon V \to W$. As we just saw, $R_G(\Opn)(W)^{\Gamma_\phi}$ is weakly contractible, so any map $\alpha \colon \partial D^l \to R_G(\Opn)(V)^{\Gamma_\phi}$ can be filled with a map $\lambda \colon D^l \to R_G(\Opn)(W)^{\Gamma_\phi}$ such that $\lambda \circ i_l = R_G(\Opn)(\psi)^{\Gamma_\phi} \circ \alpha$.

We check that $R_G(\Op)$ is a fibrant symmetric sequence in $\Spc$ (in the sense of Remark~\ref{remmodelssfib}, objectwise fibrant). We first note that if $\Gamma_\phi \notin \FRGOpKn$, we saw earlier in this proof that $R_G(\Opn)(W)^{\Gamma_\phi} = \emptyset$ for each $W$. If $\Gamma_\phi \in \FRGOpKn$, then we just saw that $R_G(\Opn)(W)^{\Gamma_\phi}$ is weakly contractible for any faithful $W$. Therefore for any embedding $\psi \colon W \to W'$ of faithful $K$-representations the map $R_G(\Opn)(\psi)^{\Gamma_\phi}$ is a weak homotopy equivalence.
\end{proof}

We recall here the notion of a $G$-transfer system for $G$ a compact Lie group, see \cite[Definition 3.4]{Rubin2} for the definition for finite groups. Let $\Sub(G)$ be the set of closed subgroups of $G$. A $G$-transfer system is a partial order on $\Sub(G)$ which refines the closed finite index subgroup relation and is closed under conjugation and restriction to subgroups. This definition is equivalent to that of a $G$-indexing system (see Remark~\ref{remindextransfer}).

Recall from the introduction the strategy of the proof of our main result. Blumberg and Hill originally conjectured that for a finite group $G$ there is an $\Nin$-operad representing each $G$-indexing system. This conjecture was separately proven by Rubin in \cite{rubin1}, Gutiérrez and White in \cite{GutierrezWhite}, and Bonventre and Pereira in \cite{BonventrePereira}, using three different methods. Gutiérrez and White proved it for $G$ a compact Lie group and not just a finite group. We will use their result to construct examples of $\Nin$-operads for a single compact Lie group $G$, then we will apply the functor $R_G$ to obtain global $\Nin$-operads with global transfer systems prescribed by Proposition~\ref{propGtoglobalnin}, and we then take products of these. In this way, we will construct in Theorem~\ref{thmmain} a global $\Nin$-operad representing any global transfer system.

First we give for $G$ a compact Lie group, an adjunction between $\Igl$, the poset of global transfer systems under inclusion, and $\IG$, the poset of $G$-transfer systems under inclusion. Note that transfer systems are partial orders themselves (on $\Lie$ or on $\Sub(G))$, and the order in the sets of transfer systems is the inclusion relation, $T \subset T'$ if $K \leqslant_T H$ implies that $K \leqslant_{T'} H$.

\begin{prop}
\label{propadjunctionG}
Let $G$ be a compact Lie group. The functor $R_G \colon \IG \to \Igl$ given by
\[K \leqslant_{R_G(T)} H \; \text{if} \quad K \leqslant H \; \text{is closed and of finite index and} \; \forall \theta \colon  L \to H \; \text{with} \; L \leqslant G , \; \theta^{-1}(K) \leqslant_T L\]
is right adjoint to the functor $U_G \colon \Igl \to \IG$ which sends a global transfer system $I'$ to its underlying $G$-transfer system, which is given by the restriction from $\Lie$ to $\Sub(G)$.
\end{prop}
\begin{proof}
It follows from Definition~\ref{defitransfer} and the definition of a $G$-transfer system that $R_G(T)$ is indeed a global transfer system, and $U_G(T')$ is a $G$-transfer system. By construction the functions $R_G$ and $U_G$ preserve inclusions and thus are functors between the respective posets.

Let $T$ be a $G$-transfer system, and let $T'$ be a global transfer system. Assume that $U_G(T') \subset T$. Let $K \leqslant_{T'} H$, since $T'$ is a global transfer system, for each $L \leqslant G$ and $\theta \colon L \to H$ we have that $\theta^{-1}(K) \leqslant_{T'} L$, and therefore the transfer from $\theta^{-1}(K)$ to $L$ is admissible in $U_G(T') \subset T$. This means that $K \leqslant_{R_G(T)} H$ by the definition of $R_G(T)$, and thus $T' \subset R_G(T)$.

Assume now that $T' \subset R_G(T)$. Let $K \leqslant L \leqslant G$, if the transfer from $K$ to $L$ is admissible for $U_G(T')$ then it is also admissible for $T'$ and thus for $R_G(T)$. Taking $\theta$ in the definition of $R_G(T)$ to be the $\id_L$ shows that $K \leqslant_T L$.
\end{proof}

\begin{thm}
\label{thmmain}
The functor $\mathpzc{T} \colon \Ho(\Ningl) \to \Igl$ from the homotopy category of global $\Nin$-operads to the poset of global transfer systems is essentially surjective, and therefore by Proposition~\ref{prophoglposet} it is an equivalence of categories.
\end{thm}
\begin{proof}
Recall that in Proposition~\ref{prophoglposet} we proved that $\Ho(\Ningl)$ is a thin category and the functor $\mathpzc{T}$ is fully faithful. Let $T \in \Igl$ be a global transfer system. For each $n \in \NN$, take the orthogonal group $O(n)$. By \cite[Theorem 4.7]{GutierrezWhite} there exists an $\Nin$-operad $\Op_{O(n)}$ with associated $O(n)$-transfer system $U_{O(n)}(T)$. The functor 
\[R_{O(n)} \colon \OnTopcat \to \Spc\]
is strong monoidal with respect to the categorical products. Therefore applying this functor to $\Op_{O(n)}$ yields an operad in $(\Spc, \times)$. The natural transformation $\rho \colon \boxtimes \Rightarrow \times$ of \cite[Section 1.3]{global} turns this into an operad $R_{O(n)}(\Op_{O(n)})$ with respect to the box product (see also \cite[Remark 2.8]{globaloperads}). By Proposition~\ref{propGtoglobalnin} the operad $R_{O(n)}(\Op_{O(n)})$ is a global $\Nin$-operad and its associated global transfer system is precisely $R_{O(n)}(U_{O(n)}(T))$. This justifies the notation we used for the functor on transfer systems $R_G \colon \IG \to \Igl$.

More explicitly, by Proposition~\ref{propGtoglobalnin} the transfer from $K$ to $H$ is admissible for $R_{O(n)}(\Op_{O(n)})$ if and only if for each $\theta \colon L \to H$ with $L \leqslant O(n)$ the graph subgroup associated to the finite $L$-set $\theta^\ast(H/K)$ is admissible for $\Op_{O(n)}$. The finite $L$-set $\theta^\ast(H/K)$ is a disjoint union of $L$-orbits of the form $L/(\theta^{-1}(hKh^{-1}))$ for $h \in H$. By the equivalence between $G$-transfer systems and $G$-indexing systems, or more concretely \cite[Lemma 4.10]{BLUMBERG2015658} and \cite[Lemma 4.15]{BLUMBERG2015658}, $\theta^\ast(H/K)$ is admissible for $\Op_{O(n)}$ if and only if each of these $L$-orbits is admissible for the $G$-transfer system associated to $\Op_{O(n)}$. Therefore the transfer from $K$ to $H$ is admissible for $R_{O(n)}(\Op_{O(n)})$ if and only if for each $\theta' \colon L \to H$ with $L \leqslant O(n)$ the transfer from $\theta'^{-1}(K)$ to $L$ is admissible for $\Op_{O(n)}$, i. e. $\theta'^{-1}(K) \leqslant_{U_{O(n)}(T)} L$, and by construction this holds if and only if $K \leqslant_{R_{O(n)}(U_{O(n)}(T))} H$. This works because we are quantifying over all $\theta'$, which includes the $\theta$ from earlier in the paragraph and all its conjugates.

By adjointness, $T \subset R_{O(n)}(U_{O(n)}(T))$ for each $n \geqslant 1$. Therefore $T \subset \bigcap_{n \geqslant 1} R_{O(n)}(U_{O(n)}(T))$. Assume that the transfer from $K$ to $H$ is admissible for the global transfer system $\bigcap_{n \geqslant 1} R_{O(n)}(U_{O(n)}(T))$. The compact Lie group $H$ embeds into some orthogonal group $O(n_0)$, and so by taking $\theta = \id_H$ we see that the transfer from $K$ to $H$ is admissible for $U_{O(n_0)}(T)$, and thus for $T$. This proves that $\bigcap_{n \in \NN} R_{O(n)}(U_{O(n)}(T)) = T$.

The category of operads in $(\Spc, \boxtimes)$ is complete and cocomplete, see \cite[\nopp 3.1.6]{fresse} for example. Proposition~\ref{propGtoglobalnin} also proves that the underlying symmetric sequence of each $R_{O(n)}(\Op_{O(n)})$ is fibrant, and thus by Proposition~\ref{proparbiproductsobspc} the operad $\prod_{n \geqslant 1} R_{O(n)}(\Op_{O(n)})$ is a global $\Nin$-operad with associated global transfer system $\bigcap_{n \geqslant 1} R_{O(n)}(U_{O(n)}(T)) = T$.
\end{proof}

In the preceding proof it is enough to consider just orthogonal groups because the restriction of a global transfer system to an $O(n)$ transfer system contains information on which transfers are admissible for all subgroups of $O(n)$. Since each compact Lie group is isomorphic to a subgroup of an orthogonal group, doing this for all $n \in \NN$ gives all the information about the global transfer system. The proof would still work if we restricted this process to groups in a set of representatives of isomorphism classes of all compact Lie groups.

\begin{rem}
\label{remEinfty}
There is a greatest global transfer system, which we denote by $\All$ and call the \emph{complete global transfer system}, with all finite index transfers admissible. The global $\Nin$-operads associated to the complete global transfer system are the global $E_\infty$-operads of \cite[Definition 5.1]{globaloperads}. There is also a smallest global transfer system, denoted by $\None$, with no transfers.

Let $G = e$ be the trivial group. Then the functor $R_e \colon \Ie \to \Igl$ sends the unique $e$-transfer system to $\All$. By the previous results, this means in particular that for any $E_\infty$-operad in spaces $\Op$ the global $\Nin$-operad $R_e(\Op)$ is actually a global $E_\infty$-operad. Note that the global operad obtained from giving $\Op$ the trivial global structure is instead a global $\Nin$-operad over the smallest global transfer system $\None$ (see \cite[Remark 5.9]{globaloperads}).

Similarly, for an arbitrary $G$, the functor $R_G \colon \IG \to \Igl$ sends the complete $G$-transfer system (which we also denote by $\All$) to the complete global transfer system $\All$. Therefore for any complete $\Nin$-operad $\Op$ for $G$ (usually called a $G$-$E_\infty$-operad) the global $\Nin$-operad $R_G(\Op)$ is also a global $E_\infty$-operad.
\end{rem}

\begin{rem}[Global $\Nin$-algebras]
If $X$ is a $G$-space which is an algebra over a $G$-equivariant $\Nin$-operad $\Op$ then since the functor $R_G$ is lax monoidal the orthogonal space $R_G(X)$ is an algebra over the global $\Nin$-operad $R_G(\Op)$. These orthogonal spaces exhibit intermediate levels of commutativity if the global transfer system associated to $R_G(\Op)$ is not the maximal one.
\end{rem}

\section{Relating global transfer systems to \texorpdfstring{$G$}{G}-transfer systems}

Our goal in this section is to study the relation between global transfer systems and transfer systems for a single compact Lie group $G$.

To start, we will restrict the class of $G$-transfer systems that we are interested in. Recall that one important conceptual difference is that global transfer systems have to be closed under restriction along arbitrary homomorphisms, but $G$-transfer systems for a fixed group $G$ only have to be closed under restriction along subgroup inclusions and conjugation. This means that for each $G$ the underlying $G$-transfer system of a global transfer system is also closed under restriction along arbitrary homomorphisms.

\begin{defi}
Let $T$ be a $G$-transfer system. We say that $T$ is \emph{closed under arbitrary homomorphisms} if for any subgroups $H, L \leqslant G$ and any continuous homomorphism $\theta \colon L \to H$, if $K \leqslant_T H$ then $\theta^{-1}(K) \leqslant_T L$.

Let $\IGc \leqslant \IG$ denote the subposet of those $G$-transfer systems closed under arbitrary homomorphisms.
\end{defi}

The functors $R_G \colon \IG \to \Igl$ and $U_G \colon \Igl \to \IG$ of Proposition~\ref{propadjunctionG} are adjoint. As mentioned at the beginning of this section, $U_G \colon \Igl \to \IG$ actually lands in $\IGc$. From the construction of $R_G$ we see that when restricted to $\IGc$, the composition $U_G \circ R_G$ is actually the identity, so in particular $U_G \colon \Igl \to \IGc$ is surjective, and $R_G$ gives an embedding of $\IGc$ into $\Igl$. Thus, in order to study the poset of global transfer systems (or equivalently, the homotopy theory of global $\Nin$-operads) we want to study the posets $\IGc$ for varying $G$.

We begin with a simple example.

\begin{ex}
Consider the cyclic group $G = C_p$ for $p$ a prime. Since it has only the trivial subgroup and $C_p$, there are exactly two different $C_p$-transfer systems, the complete one, which we denote by $\All$, and the trivial transfer system without the transfer from $\langle e\rangle$ to $C_p$, which we denote by $\None$. Both these transfer systems are in $\overline{I}_{C_p}$.

By Remark~\ref{remEinfty}, $R_{C_p}(\All) = \All$. However the global transfer system $R_{C_p}(\None)$ is our first example that is neither complete nor trivial. For a compact Lie group $H$ and $K\leqslant H$ of finite index, since homomorphisms from $C_p$ to $H$ are equivalent to elements of $H$ of order $p$, the transfer from $K$ to $H$ is admissible in $R_{C_p}(\None)$ if and only if each element $h \in H$ of order $p$ belongs to $K$.
\end{ex}

More generally, If $G$ is any compact Lie group, then there is a complete $G$-transfer system denoted by $\All$, which is the maximal element of $\IG$, and a trivial $G$-transfer system with no transfers, denoted by $\None$, which is the minimal element of $\IG$. It is simple to check that they are both in $\IGc$.

Completely computing the poset $\IG$ of $G$-transfer systems is in general a very hard task, and it has only been done for some concrete finite groups (see \cite{balchin2021ninftyoperads}, \cite{Rubin2}, \cite{balchinCpqr}). However the subposet $\IGc$ of $G$-transfer systems closed under arbitrary homomorphisms is much easier to compute, as we will show through the following examples.

\begin{ex}
Balchin, Barnes and Roitzheim prove in \cite{balchin2021ninftyoperads} that for $p$ prime and $n \geqslant 1$ the poset $I_{C_{p^n}}$ of transfer systems for $C_{p^n}$ is isomorphic to the $(n + 1)$-associahedron. In particular, it has cardinality the Catalan number
\[\catalan(n + 1) = \frac{(2n + 2)!}{(n + 2)! n!}.\]

For each $0 \leqslant m \leqslant n$ let $I_m$ be the $C_{p^n}$-transfer system that contains the transfer  $C_{p^i} \leqslant C_{p^j}$ with $i < j$ if and only if $i \geqslant n - m$. Note that $I_n = \All$ and $I_0 = \None$. Using the description of $C_{p^n}$-transfer systems as graphs used throughout \cite{balchin2021ninftyoperads} it is possible to check that these are the only transfer systems closed under arbitrary homomorphisms, as this implies that if the transfer $C_{p^i} \leqslant C_{p^{i + 1}}$ is admissible, then so is the transfer $C_{p^{i + j}} \leqslant C_{p^{i + j + 1}}$ for any $j \leqslant n - i - 1$. Thus the poset $\overline{I}_{C_{p^n}}$ of transfer systems coming from global transfer systems has $n + 1$ elements and is linearly ordered.
\end{ex}

\begin{ex}
We will now consider $G = \Sigma_3$ the symmetric group of order $6$. Note that $I_{\Sigma_3}$ has a more interesting structure than previous examples because $\Sigma_3$ is not abelian. The poset $I_{\Sigma_3}$ was computed in \cite[Figure~4]{Rubin2}. We describe it here to compute its subposet $\overline{I}_{\Sigma_3}$ of transfer systems coming from global transfer systems.

Let $\sigma$ and $\tau$ be elements of $\Sigma_3$ of orders $3$ and $2$ respectively, and let $e$ be the identity of $\Sigma_3$. A $\Sigma_3$-transfer system $T\in I_{\Sigma_3}$ is a partial order on $\Sub(\Sigma_3)$ which is closed under restriction to subgroups and conjugation. The closure under conjugation in particular means that for each $g \in G$, $K \leqslant_T H$ if and only if $g K g^{-1} \leqslant_T g H g^{-1}$. Thus $T$ is completely determined by which of the following transfers are admissible for $T$: $\langle e \rangle \leqslant \Sigma_3$, $\langle \tau \rangle \leqslant \Sigma_3$, $\langle \sigma \rangle \leqslant \Sigma_3$, $\langle e \rangle \leqslant \langle \tau \rangle$, and $\langle e \rangle \leqslant \langle \sigma \rangle$. We represent these abstract transfers in the following diagram.

\begin{equation}
\label{diagSsuborder}
\begin{tikzcd}[every label/.append style = {font = \LARGE}]
	& {\Sigma_3} \\
	{\langle \sigma \rangle} && {\langle \tau \rangle} \\
	& {\langle e \rangle}
	\arrow["\geqslant"{description}, sloped, draw=none, from=3-2, to=2-1]
	\arrow["\leqslant"{description}, sloped, draw=none, from=3-2, to=2-3]
	\arrow["\leqslant"{description}, sloped, draw=none, from=3-2, to=1-2]
	\arrow["\leqslant"{description}, sloped, draw=none, from=2-1, to=1-2]
	\arrow["\geqslant"{description}, sloped, draw=none, from=2-3, to=1-2]
\end{tikzcd}
\end{equation}

A $\Sigma_3$-transfer system $T$ being closed under restriction to subgroups means that if $K \leqslant_T H$ and $L$ is a subgroup of $H$, then $L \cap K \leqslant_T L$. This imposes the following conditions on $T$:
\begin{enumerate}[\roman*)]
    \item If $\langle e \rangle \leqslant_T \Sigma_3$ then $\langle e \rangle \leqslant_T \langle \tau \rangle$ and $\langle e \rangle \leqslant_T \langle \sigma \rangle$.
    \item If $\langle \sigma \rangle \leqslant_T \Sigma_3$ then $\langle e \rangle \leqslant_T \langle \tau \rangle$.
    \item If $\langle \tau \rangle \leqslant_T \Sigma_3$ then $\langle e \rangle \leqslant_T \langle \sigma \rangle$.
    \item \label{conju} If $\langle \tau \rangle \leqslant_T \Sigma_3$ then $\langle e \rangle \leqslant_T \langle \sigma \tau \sigma^{-1} \rangle$ and by conjugation $\langle e \rangle \leqslant_T \langle \tau \rangle$.
\end{enumerate}

It is important to remember the different conjugate subgroups of $\Sigma_3$. Even though closure under conjugation means that admissibility of conjugacy classes of pairs $K \leqslant H$ completely determines a $\Sigma_3$-transfer system, different conjugate subgroups can still impose restrictions on which transfers are admissible, as in \ref{conju}.

Lastly, recall that $T$ has to be transitive as it is a partial order. These are all the conditions that a $\Sigma_3$-transfer system has to satisfy, thus we can compute the poset $I_{\Sigma_3}$, shown in figure~(\ref{diagS3}) below.

The $\Sigma_3$-transfer systems in $\overline{I}_{\Sigma_3}$ have to additionally be closed under arbitrary homomorphisms. These homomorphisms can be decomposed as an epimorphism followed by an isomorphism followed by a subgroup inclusion. All isomorphisms of subgroups of $\Sigma_3$ are induced by conjugation by an element of $\Sigma_3$, so additional restrictions on $T$ will only come from non-injective epimorphisms, and the only one with non-trivial image is $\theta \colon \Sigma_3 \to \langle \tau \rangle$ (up to conjugation). This epimorphism $\theta$ creates the additional condition on $T \in \overline{I}_{\Sigma_3}$ that if $\langle e \rangle \leqslant_T \langle \tau \rangle$ then $\langle \sigma \rangle \leqslant_T \Sigma_3$.

Below we show the poset $I_{\Sigma_3}$ and its subposet $\overline{I}_{\Sigma_3}$ in the style of the figures of \cite{Rubin2}. Each of the nodes of a diagram represents a partial order on $\Sub(\Sigma_3)$ with the relations given by the edges of the node, and with subgroups distributed as in diagram (\ref{diagSsuborder}). The dashed lines represent the partial orders on $I_{\Sigma_3}$ and $\overline{I}_{\Sigma_3}$ given by inclusion, from the lower to the higher nodes on the picture.

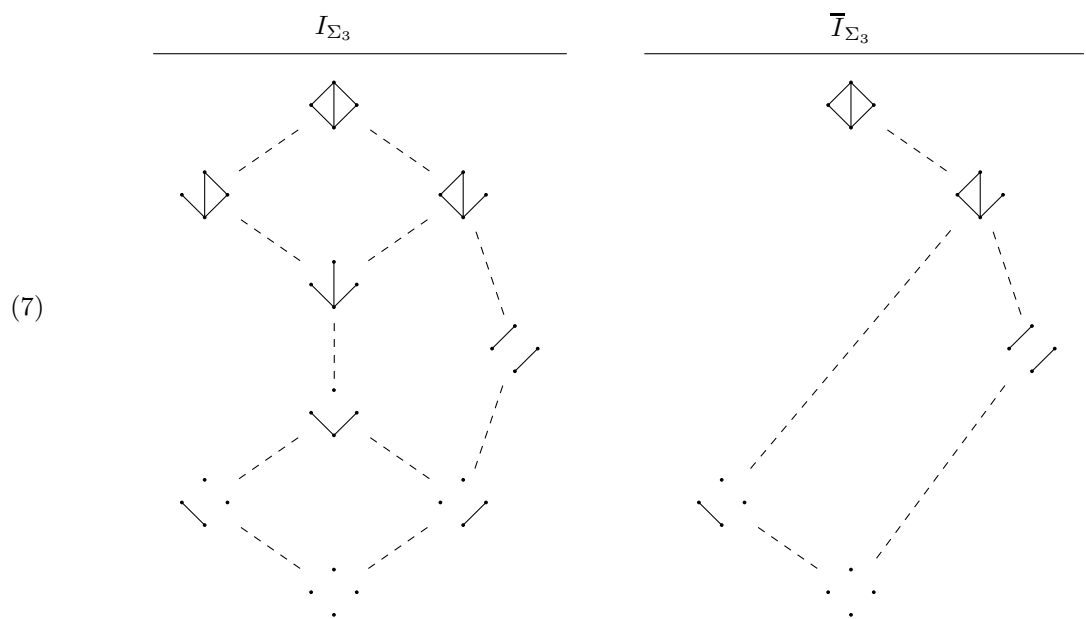
\begin{equation}
\label{diagS3}
\begin{tikzpicture}[scale=1.7, baseline=(current  bounding  box.center)]
	\node(A) at (4,0) {$\trSempty$};
	\node(B) at (3,0.7) {$\trSsigma$};
	\node(C) at (5,0.7) {$\trStau$};
	\node(D) at (4,1.4) {$\trStausigma$};
	\node(E') at (4,2.4) {$\trStausigmaS$};
	\node(F) at (5.4,1.9) {$\trSsigmatoS$};
	\node(G) at (5,3.1) {$\trSSsigmatoS$};
	\node(I) at (3,3.1) {$\trSStautoS$};
	\node(J) at (4,3.8) {$\trSall$};
	
	\node(K) at (8,0) {$\trSempty$};
	\node(L) at (7,0.7) {$\trSsigma$};
	\node(M) at (9.4,1.9) {$\trSsigmatoS$};
	\node(P) at (9,3.1) {$\trSSsigmatoS$};
	\node(O) at (8,3.8) {$\trSall$};

	\path[dashed]
	(A) edge node {} (B)
	
	(A) edge node {} (C)
	(B) edge node {} (D)
	(C) edge node {} (D)
	
	(D) edge node {} (E')
	
	(C) edge node {} (F)
	(F) edge node {} (G)
	(E') edge node {} (G)
	
	(E') edge node {} (I)
	(I) edge node {} (J)
	(G) edge node {} (J)
	
	(K) edge node {} (L)
	(K) edge node {} (M)
	(L) edge node {} (P)
	(M) edge node {} (P)
	(P) edge node {} (O)
	;
	
	\node at (4,4.4) {$I_{\Sigma_3}$};
	\node at (8,4.4) {$\overline{I}_{\Sigma_3}$};
	\draw (2.6,4.2) -- (5.8,4.2) ;
	\draw (6.4,4.2) -- (9.8,4.2) ;
\end{tikzpicture}\end{equation}
\end{ex}

\begin{ex}
Let $Q_8$ denote the quaternion group. The poset $I_{Q_8}$ was computed in \cite[Figure~3]{Rubin2}, and it consists of $68$ elements. However the subposet $\overline{I}_{Q_8}$ contains just $6$ elements. The main reason for this greater disparity compared to the case of $\Sigma_3$ is that the three\footnote{The author mistakenly wrote \emph{four} in the published version.} subgroups of $Q_8$ of order 4 are normal, so conjugation does not force any admissibility relations between the transfers associated to them for general $Q_8$-transfer systems. However the outer automorphisms of $Q_8$ provide such relations for $Q_8$-transfer systems closed under arbitrary homomorphisms, causing the subposet $\overline{I}_{Q_8}$ to be much smaller.
\end{ex}

To summarize, from the examples of this section we can deduce that the additional structure on the category of global $\Nin$-operads imposed by the global compatibility of the actions for all groups substantially limits the possible underlying $G$-homotopy types of global $\Nin$-operads, compared to all $\Nin$-operads for the group $G$. However the collection of possible global $\Nin$-operads shows a complex and very interesting structure, which we analyze further in upcoming work.

\appendix
\section{A quick note on definitions of operads}
\label{appendixdefop}

Let $(\cat, \otimes, \un)$ be a cocomplete symmetric monoidal category, where the tensor product preserves all small colimits in both variables. There are (at least) two different but equivalent definitions for the concept of an operad in $(\cat, \otimes, \un)$. One is named \emph{monoidal definition} in \cite{LodayVallette}. According to it an operad is a monoid in $(\sob, \circ)$, the category of symmetric sequences in $\cat$ with a particular monoidal structure, the monoidal composition structure (see \cite[\nopp 2.2.2]{fresse}). This is the definition used in this paper. A different one, named the \emph{classical definition} in \cite{LodayVallette}, states that an operad is a symmetric sequence together with a unit and a set of \emph{composition morphisms} that satisfy some associativity, unital and equivariance axioms. This is the original definition given by May in \cite{may1972geometry}.

These two definitions are equivalent. This can be deduced from putting together \cite[\nopp 2.1.8]{fresse}, \cite[\nopp 2.2.2]{fresse} and \cite[\nopp 3.2.11]{fresse}. We will expand this slightly, but we will only give full details for the specific statement that we need in the body of this paper (given as Lemma~\ref{lemmdefsoperad}). For $\Op \in \sob$ a symmetric sequence, by the previously mentioned results from \cite{fresse}, we have the following description
\begin{equation}
\label{eqappendix}
    (\Op \circ \Op)_n = \coprod_{k \in \NN} \Op_k \otimes_{\Sigma_k} (\coprod_{n_1 + \dots + n_k = n} (\Sigma_n \otimes_{\Sigma_{n_1} \times \dots \times \Sigma_{n_k}} \Op_{n_1} \otimes \dots \otimes \Op_{n_k}))= \coprod_{k \in \NN} (\Op \circ \Op)_{n, k}.
\end{equation}

We introduced the notation $(\Op \circ \Op)_{n, k}$ for each of the terms of the coproduct indexed by $\NN$. The $\Sigma_k$-action on the coproduct indexed by the tuples $n_1 + \dots + n_k = n$ is the following. Consider a permutation $\sigma \in \Sigma_k$ and $n_1, \dots , n_k \geqslant 0$ such that $n_1 + \dots + n_k = n$. Let $\blk(\sigma, n_1, \dots , n_k)~\in~\Sigma_n$ denote the block permutation obtained by permuting blocks of sizes $n_1, \dots , n_k$ using $\sigma$. The element $\sigma$ acts by permuting the terms $\Op_{n_1} \otimes \dots \otimes \Op_{n_k}$, and by left multiplication with $\blk(\sigma, n_1, \dots , n_k)$ in $\Sigma_n$.

An operad structure on $\Op$ includes a composition morphism of symmetric sequences $\mu \colon \Op \circ \Op \to \Op$. By (\ref{eqappendix}) this morphism is the same as a set of morphisms
\[\mu_{k, n_1, \dots, n_k} \colon \Op_k \otimes \Op_{n_1} \otimes \dots \otimes \Op_{n_k} \to \Op_{n_1 + \dots + n_k}\]
satisfying some equivariance conditions. For the purposes of Lemma~\ref{lemmdisjointunion} and Proposition~\ref{propopgivesindexing} we would like to write out these equivariance conditions as saying that the morphisms $\mu_{k, n_1, \dots, n_k}$ are $(\Sigma_{n_1} \times \dots \times \Sigma_{n_k})$-equivariant and $\Sigma_k$-equivariant for some well-defined actions. This is however not possible, due to the fact that acting via $\blk(\sigma, n_1, \dots , n_k)$ on $\Op_{n_1 + \dots + n_k}$ does \underline{\emph{not}} define a $\Sigma_k$-action because $\blk(-, n_1, \dots , n_k) \colon \Sigma_k \to \Sigma_{n_1 + \dots + n_k}$ is not generally a group homomorphism. However if for example $n_1 = \dots = n_k$ then $\blk(-, n_1, \dots , n_k)$ does indeed describe a group homomorphism: the composition $\Sigma_k \to \Sigma_k \wr \Sigma_{n'} \to \Sigma_{k n'}$. The following lemma is a generalization of this observation and it is enough for our purposes.

\begin{lemm}
\label{lemmdefsoperad}
Let $\Op$ be an operad in $\cat$ with respect to the monoidal definition of \cite[\nopp 3.1.1]{fresse}, with composition $\mu \colon \Op \circ \Op \to \Op$. Then for each $k, n_1, \dots, n_k \geqslant 0$ the composition $\mu$ determines a morphism in $\cat$
\[\mu_{k, n_1, \dots, n_k} \colon \Op_k \otimes \Op_{n_1} \otimes \dots \otimes \Op_{n_k} \to \Op_{n_1 + \dots + n_k}.\]

Consider $k_1, \dots, k_p \geqslant 1$ such that $k_1 + \dots + k_p=k$. Consider also $m_1, \dots, m_p \geqslant 0$. Let $\underline{k}$ denote the tuple of $k$ natural numbers formed by $k_1$ copies of $m_1$, $k_2$ copies of $m_2$ and so on. Let $m = \Sigma_{i=1}^p k_i p_i$. Consider the group 
\[G=(\Sigma_{k_1} \times \dots \times \Sigma_{k_p}) \ltimes (\Sigma_{m_1}^{\times k_1} \times \dots \times \Sigma_{m_p}^{\times k_p})\]
and the homomorphism $\phi \colon G \to \Sigma_m$ given by sending an element of $\Sigma_{k_1} \times \dots \times \Sigma_{k_p}$ to the associated block permutation in $\Sigma_m$, and an element of $\Sigma_{m_1}^{\times k_1} \times \dots \times \Sigma_{m_p}^{\times k_p}$ to the permutation of $\Sigma_m$ that acts individually on each block.

Then the previously mentioned morphism 
\[\mu_{k, \underline{k}} \colon \Op_k \otimes \Op_{m_1}^{\otimes k_1} \otimes \dots \otimes \Op_{m_p}^{\otimes k_p} \to \Op_m\]
is $G$-equivariant. Here $G$ acts on $\Op_m$ through $\phi$, each term $\Sigma_{m_i}$ in the definition of $G$ acts on $\Op_{m_i}$, and the term $\Sigma_{k_1} \times \dots \times \Sigma_{k_p} \leqslant \Sigma_k$ acts on $\Op_k$ and it acts on $\Op_{m_1}^{\otimes k_1} \otimes \dots \otimes \Op_{m_p}^{\otimes k_p}$ by permuting the factors.

\end{lemm}
\begin{proof}
From \cite[\nopp 2.1.8]{fresse} and \cite[\nopp 2.2.2]{fresse} we obtain (\ref{eqappendix}), which produces the morphisms $\mu_{k, n_1, \dots, n_k}$ from the $\Sigma_n$-equivariant morphisms $\mu_n \colon (\Op \circ \Op)_n \to \Op_n$. The morphisms $\mu_{k, n_1, \dots, n_k}$ are $(\Sigma_{n_1} \times \dots \times \Sigma_{n_k})$-equivariant because induction $\Sigma_n \otimes_{\Sigma_{n_1} \times \dots \times \Sigma_{n_k}} -$ is left adjoint to restriction of the action.

The set function $\blk(-, \underline{k}) \colon \Sigma_k \to \Sigma_m$ is not a group homomorphism, however due to the form of the tuple of indices $\underline{k}$, it is a group homomorphism if we restrict to $\Sigma_{k_1} \times \dots \times \Sigma_{k_p} \leqslant \Sigma_k$, and this induces the homomorphism $\phi \colon G \to \Sigma_m$. We also obtain that the actions mentioned in the statement of this lemma are well-defined group actions.

The $(\Sigma_{k_1} \times \dots \times \Sigma_{k_p})$-equivariance of the morphism $\mu_{k, \underline{k}}$ with the group actions given in the statement of this lemma follows from the $\Sigma_m$-equivariance of $\mu_m \colon (\Op \circ \Op)_m \to \Op_m$ and the description of the $\Sigma_k$-action on (\ref{eqappendix}) given earlier.
\end{proof}

Note that this lemma does not include all the equivariance requirements on the morphisms $\mu_{k, n_1, \dots, n_k}$ that appear in the classical definition of operad.

\printbibliography

@article {BLUMBERG2015658,
    AUTHOR = {Blumberg, Andrew J. and Hill, Michael A.},
     TITLE = {Operadic multiplications in equivariant spectra, norms, and
              transfers},
   JOURNAL = {Adv. Math.},
  FJOURNAL = {Advances in Mathematics},
    VOLUME = {285},
      YEAR = {2015},
     PAGES = {658--708},
      ISSN = {0001-8708},
   MRCLASS = {55P91 (18D50 55P43 55P48)},
  MRNUMBER = {3406512},
MRREVIEWER = {Markus Szymik},
       DOI = {10.1016/j.aim.2015.07.013},
       URL = {https://doi.org/10.1016/j.aim.2015.07.013},
}

@book {global,
    AUTHOR = {Schwede, Stefan},
     TITLE = {Global homotopy theory},
    SERIES = {New Mathematical Monographs},
    VOLUME = {34},
 PUBLISHER = {Cambridge University Press, Cambridge},
      YEAR = {2018},
     PAGES = {xviii+828},
      ISBN = {978-1-108-42581-0},
   MRCLASS = {55P42 (18G55 19D99 55P91 55U35)},
  MRNUMBER = {3838307},
MRREVIEWER = {Gregory Z. Arone},
       DOI = {10.1017/9781108349161},
       URL = {https://doi.org/10.1017/9781108349161},
}

@book {hovey2007model,
    AUTHOR = {Hovey, Mark},
     TITLE = {Model categories},
    SERIES = {Mathematical Surveys and Monographs},
    VOLUME = {63},
 PUBLISHER = {American Mathematical Society, Providence, RI},
      YEAR = {1999},
     PAGES = {xii+209},
      ISBN = {0-8218-1359-5},
   MRCLASS = {55U35 (18D15 18G30 18G55)},
  MRNUMBER = {1650134},
MRREVIEWER = {Teimuraz Pirashvili},
}

@book {fresse,
    AUTHOR = {Fresse, Benoit},
     TITLE = {Modules over operads and functors},
    SERIES = {Lecture Notes in Mathematics},
    VOLUME = {1967},
 PUBLISHER = {Springer-Verlag, Berlin},
      YEAR = {2009},
     PAGES = {x+308},
      ISBN = {978-3-540-89055-3},
   MRCLASS = {18D50 (18G50 55P48 57T30)},
  MRNUMBER = {2494775},
MRREVIEWER = {Paul G. Goerss},
       DOI = {10.1007/978-3-540-89056-0},
       URL = {https://doi.org/10.1007/978-3-540-89056-0},
}

@book {diecklie,
    AUTHOR = {Br\"{o}cker, Theodor and tom Dieck, Tammo},
     TITLE = {Representations of compact {L}ie groups},
    SERIES = {Graduate Texts in Mathematics},
    VOLUME = {98},
 PUBLISHER = {Springer-Verlag, New York},
      YEAR = {1985},
     PAGES = {x+313},
      ISBN = {0-387-13678-9},
   MRCLASS = {22E45 (57-01)},
  MRNUMBER = {781344},
MRREVIEWER = {R. Schultz},
       DOI = {10.1007/978-3-662-12918-0},
       URL = {https://doi.org/10.1007/978-3-662-12918-0},
}

@book {may1972geometry,
    AUTHOR = {May, J. Peter},
     TITLE = {The geometry of iterated loop spaces},
    SERIES = {Lecture Notes in Mathematics, Vol. 271},
 PUBLISHER = {Springer-Verlag, Berlin-New York},
      YEAR = {1972},
     PAGES = {viii+175},
   MRCLASS = {55D35},
  MRNUMBER = {0420610},
MRREVIEWER = {J. Stasheff},
}

@unpublished{globaloperads,
      title={Operads in unstable global homotopy theory}, 
      author={Miguel Barrero},
      year={2021},
    note={arXiv:2110.01674 Accepted for publication in Algebraic \& Geometric Topology},
    eprint={2110.01674},
    archivePrefix={arXiv},
    primaryClass={math.AT}
}

@article {abeliantransfers,
    AUTHOR = {Barrero, Miguel},
     TITLE = {Global transfer systems of abelian compact Lie groups},
   JOURNAL = {Proc. Amer. Math. Soc.},
  FJOURNAL = {Proceedings of the American Mathematical Society},
    VOLUME = {151},
      YEAR = {2023},
    NUMBER = {7},
     PAGES = {3169--3182},
      ISSN = {0002-9939},
   MRCLASS = {55P91 (06A15 18M60)},
  MRNUMBER = {4579387},
       DOI = {10.1090/proc/16310},
       URL = {https://doi.org/10.1090/proc/16310},
}

@article {Rubin2,
    AUTHOR = {Rubin, Jonathan},
     TITLE = {Detecting {S}teiner and linear isometries operads},
   JOURNAL = {Glasg. Math. J.},
  FJOURNAL = {Glasgow Mathematical Journal},
    VOLUME = {63},
      YEAR = {2021},
    NUMBER = {2},
     PAGES = {307--342},
      ISSN = {0017-0895},
   MRCLASS = {55P91 (55P48)},
  MRNUMBER = {4244201},
       DOI = {10.1017/S001708952000021X},
       URL = {https://doi.org/10.1017/S001708952000021X},
}

@article {balchin2021ninftyoperads,
    AUTHOR = {Balchin, Scott and Barnes, David and Roitzheim, Constanze},
     TITLE = {$N_\infty$-operads and associahedra},
   JOURNAL = {Pacific J. Math.},
  FJOURNAL = {Pacific Journal of Mathematics},
    VOLUME = {315},
      YEAR = {2021},
    NUMBER = {2},
     PAGES = {285--304},
      ISSN = {0030-8730},
   MRCLASS = {18M80 (06A07 52B20 55N91 55P91)},
  MRNUMBER = {4366744},
       DOI = {10.2140/pjm.2021.315.285},
       URL = {https://doi.org/10.2140/pjm.2021.315.285},
}

@article {balchinCpqr,
    AUTHOR = {Scott Balchin and Daniel Bearup and Clelia Pech and Constanze Roitzheim},
     TITLE = {Equivariant homotopy commutativity for $G=C_{pqr}$},
   JOURNAL = {Tbilisi Mathematical Journal},
    VOLUME = {Special Issue on Homotopy Theory, Spectra, and Structured Ring Spectra},
      YEAR = {2020},
     PAGES = {17--31}
}

@article {rubin1,
    AUTHOR = {Rubin, Jonathan},
     TITLE = {Combinatorial {$N_\infty$} operads},
   JOURNAL = {Algebr. Geom. Topol.},
  FJOURNAL = {Algebraic \& Geometric Topology},
    VOLUME = {21},
      YEAR = {2021},
    NUMBER = {7},
     PAGES = {3513--3568},
      ISSN = {1472-2747},
   MRCLASS = {55P48 (55P91)},
  MRNUMBER = {4357612},
       DOI = {10.2140/agt.2021.21.3513},
       URL = {https://doi.org/10.2140/agt.2021.21.3513},
}

@book {LodayVallette,
    AUTHOR = {Loday, Jean-Louis and Vallette, Bruno},
     TITLE = {Algebraic operads},
    SERIES = {Grundlehren der mathematischen Wissenschaften [Fundamental
              Principles of Mathematical Sciences]},
    VOLUME = {346},
 PUBLISHER = {Springer, Heidelberg},
      YEAR = {2012},
     PAGES = {xxiv+634},
      ISBN = {978-3-642-30361-6},
   MRCLASS = {18D50 (16E99)},
  MRNUMBER = {2954392},
MRREVIEWER = {Andrey Yu. Lazarev},
       DOI = {10.1007/978-3-642-30362-3},
       URL = {https://doi.org/10.1007/978-3-642-30362-3},
}

@article {GutierrezWhite,
    AUTHOR = {Guti\'{e}rrez, Javier J. and White, David},
     TITLE = {Encoding equivariant commutativity via operads},
   JOURNAL = {Algebr. Geom. Topol.},
  FJOURNAL = {Algebraic \& Geometric Topology},
    VOLUME = {18},
      YEAR = {2018},
    NUMBER = {5},
     PAGES = {2919--2962},
      ISSN = {1472-2747},
   MRCLASS = {55P42 (55P48 55P60 55P91 55U35)},
  MRNUMBER = {3848404},
MRREVIEWER = {Steffen Sagave},
       DOI = {10.2140/agt.2018.18.2919},
       URL = {https://doi.org/10.2140/agt.2018.18.2919},
}

@article {BonventrePereira,
    AUTHOR = {Bonventre, Peter and Pereira, Lu\'{i}s A.},
     TITLE = {Genuine equivariant operads},
   JOURNAL = {Adv. Math.},
  FJOURNAL = {Advances in Mathematics},
    VOLUME = {381},
      YEAR = {2021},
     PAGES = {Paper No. 107502, 133},
      ISSN = {0001-8708},
   MRCLASS = {18M60 (55P48)},
  MRNUMBER = {4205708},
MRREVIEWER = {Lada Peksov\'{a}},
       DOI = {10.1016/j.aim.2020.107502},
       URL = {https://doi.org/10.1016/j.aim.2020.107502},
}

@article {hausmann2022global,
    AUTHOR = {Hausmann, Markus},
     TITLE = {Global group laws and equivariant bordism rings},
   JOURNAL = {Ann. of Math. (2)},
  FJOURNAL = {Annals of Mathematics. Second Series},
    VOLUME = {195},
      YEAR = {2022},
    NUMBER = {3},
     PAGES = {841--910},
      ISSN = {0003-486X},
   MRCLASS = {57R85 (14L05 55N22 55P91)},
  MRNUMBER = {4413745},
       DOI = {10.4007/annals.2022.195.3.2},
       URL = {https://doi.org/10.4007/annals.2022.195.3.2},
}

@article {kervaire,
    AUTHOR = {Hill, Michael A. and Hopkins, Michael J. and Ravenel, Douglas C.},
     TITLE = {On the nonexistence of elements of {K}ervaire invariant one},
   JOURNAL = {Ann. of Math. (2)},
  FJOURNAL = {Annals of Mathematics. Second Series},
    VOLUME = {184},
      YEAR = {2016},
    NUMBER = {1},
     PAGES = {1--262},
      ISSN = {0003-486X},
   MRCLASS = {55P91 (55N22 55P42 55Q45 55T15 55U35 57R15)},
  MRNUMBER = {3505179},
MRREVIEWER = {Paul G. Goerss},
       DOI = {10.4007/annals.2016.184.1.1},
       URL = {https://doi.org/10.4007/annals.2016.184.1.1},
}

@article {tambara,
    AUTHOR = {Blumberg, Andrew J. and Hill, Michael A.},
     TITLE = {Incomplete {T}ambara functors},
   JOURNAL = {Algebr. Geom. Topol.},
  FJOURNAL = {Algebraic \& Geometric Topology},
    VOLUME = {18},
      YEAR = {2018},
    NUMBER = {2},
     PAGES = {723--766},
      ISSN = {1472-2747},
   MRCLASS = {55P91 (18B99 19A22 55N91)},
  MRNUMBER = {3773736},
MRREVIEWER = {Lennart Meier},
       DOI = {10.2140/agt.2018.18.723},
       URL = {https://doi.org/10.2140/agt.2018.18.723},
}

@article {mackey,
    AUTHOR = {Blumberg, Andrew J. and Hill, Michael A.},
     TITLE = {Equivariant stable categories for incomplete systems of
              transfers},
   JOURNAL = {J. Lond. Math. Soc. (2)},
  FJOURNAL = {Journal of the London Mathematical Society. Second Series},
    VOLUME = {104},
      YEAR = {2021},
    NUMBER = {2},
     PAGES = {596--633},
      ISSN = {0024-6107},
   MRCLASS = {55P42 (55P91)},
  MRNUMBER = {4311105},
       DOI = {10.1112/jlms.12441},
       URL = {https://doi.org/10.1112/jlms.12441},
}

\end{document}